\documentclass[11pt,a4paper]{amsart}
\usepackage[T1]{fontenc}
\usepackage{latexsym}
\usepackage{amsfonts}
\usepackage{amssymb}
\usepackage{url}
\usepackage{mathrsfs}
\usepackage{amscd}

\usepackage{fullpage}

\def\stacksum#1#2{{\stackrel{{\scriptstyle #1}}
{{\scriptstyle #2}}}}

\newcommand{\ra}{\rightarrow}

\newcommand{\Tr}{\mathrm{Tr}}

\newcommand{\Gal}{{\rm Gal}}
\newcommand{\End}{\mathrm{End}}
\newcommand{\mt}{\mathcal}
\newcommand{\nsp}{{\rm N}_{\rm Spin}}
\newcommand{\disc}{{\rm disc}}
\newcommand{\rank}{{\rm rk}}

\newcommand{\C}{\mathbf{C}}

\newcommand{\Z}{\mathbf{Z}}

\newcommand{\R}{\mathbf{R}}

\newcommand{\Q}{\mathbf{Q}}
\newcommand{\F}{\mathbf{F}}
\newcommand{\G}{\mathbf{G}}
\newcommand{\Pp}{\mathcal{P}}
\newcommand{\Lstar}{\mt{L}^*}
\newcommand{\spin}{{\rm Spin}}

\newcommand{\eps}{\varepsilon}
\renewcommand{\epsilon}{\varepsilon}

\theoremstyle{plain}
\newtheorem{theorem}{Theorem}

\newtheorem{lemme}[theorem]{Lemma}

\newtheorem{proposition}[theorem]{Proposition}

\theoremstyle{remark}
\newtheorem*{remark}{Remark}
\newtheorem*{remarks}{Remarks}

\theoremstyle{definition}
\newtheorem{definition}[theorem]{Definition}

\begin{document}

\title[Random walks on left cosets of arithmetic groups]{The Large Sieve and random walks on left cosets of arithmetic groups}

\author{F. Jouve}
\address{Universit\'e Bordeaux I - IMB\\
351, cours de la Lib\'eration\\
33405 Talence Cedex\\
France}
\email{florent.jouve@math.u-bordeaux1.fr}


\begin{abstract}
Building on recent work of Kowalski on random walks on $SL(n,\Z)$ and $Sp(2g,\Z)$, we consider similar problems (we try to estimate the probability with which, after $k$ steps, the matrix obtained has a characteristic polynomial with maximal Galois group or has no nonzero squares among its entries) for more general classes of sets: in $GL(n,A)$, where $A$ is a subring of $\Q$ containing $\Z$ that we specify, we perform a random walk on the set of matrices with fixed determinant $D\in A^\times$. We also investigate the case where the set involved is any of the two left cosets of the special orthogonal group $SO(n,m)(\Z)$ with respect to the spinorial kernel $\Omega(n,m)(\Z)$.
\end{abstract}

\maketitle

\section*{Introduction and statement of the results}
 For $G$ a fixed subgroup of $GL(n,\Q)$, it is natural to wonder what the typical behavior of an element $g\in G$ chosen at random should be. That kind of question is investigated by Kowalski in~\cite[Chap. 7]{KoLS}. With in mind such intuitive facts as: a random element should have, with high probability, an irreducible characteristic polynomial (or indeed, one with large splitting field) and no square among its entries, Kowalski shows that the $k$-th step of a random walk lies in the set of the exceptional elements of $G$ (i.e. the elements which do not satisfy the desired property) with probability tending to zero exponentially as $k$ grows to infinity.
 \par
 In loc. cit., these results are obtained, in the case where $G=SL(n,\Z)$ or $Sp(2g,\Z)$ (for $n\geqslant 2$ and $g\geqslant 2$), as an application of the very general large sieve framework exposed in the first chapters of~\cite{KoLS}.
 \par
 \medskip
 In this paper, we answer the same type of questions (i.e. we try to detect similar properties) for sets $Y$ being either left cosets $\alpha SL(n,A)$ of $GL(n,A)$ (where $n\geqslant 2$ and $A$ is a subring of $\Q$ containing $\Z$ which we will specify) or left cosets of $SO(n,m)(\Z)$ for $n+m\geqslant 6$ (i.e. we will fix an indefinite quadratic form with signature $(n,m)$ when seen as defined over a $(n+m)$-dimensional space over $\R$) with respect to the normal subgroup $\Omega(n,m)(\Z)$ (which is to be described later). The method used is that of the ``coset sieve'' described by Kowalski in~\cite[Chap. 3.3]{KoLS} (see also~\cite{KoCrelle} where that idea already appears to study properties of the numerator of zeta functions of curves over finite fields).
 
 \par
 \medskip
 Let us now define what is needed to give the precise statements for the main results of this paper. The first kind of subgroups $G$ of $GL(n,\Q)$ we consider are of the type $G=GL(n,A)$ where, if $\Pp$ denotes a (possibly infinite) set of primes with complement having positive density, then $A$ is taken to be equal to the ring $\Z[1/\Pp]$ which is the smallest subring of $\Q$ containing $\Z$ in which every $p\in\Pp$ is invertible. The left coset to which we apply large sieve techniques in that first case is a fixed element of $GL(n,A)/SL(n,A)$.
 \par
 We also consider the case where $G$ is the subgroup of integral points of a special orthogonal group: for $n+m\geqslant 6$, let $(M,Q)$ be a quadratic module over $\Z$ such that, seen over $\R$, the quadratic form $Q$ is indefinite with signature $(n,m)$. The group of automorphisms of $M$ preserving $Q$ can be seen as the subgroup of integral points (denoted $O(n,m)(\Z)$) of the algebraic group ${\bf O}(n,m)/\Q$. We will restrict ourselves to the case where $G=SO(n,m)(\Z)$, the subgroup of integral points of the algebraic group ${\bf SO}(n,m)/\Q$. In $SO(n,m)(\Z)$ lies the normal subgroup $\Omega(n,m)(\Z)$ (see~\cite[Section 7.2C, pp 422--424]{HOM} where that subgroup is denoted $O'(M)$). A precise description for that group will be given in Section~\ref{localdensities} in the case where $M$ is a vector space over a finite field and in Section~\ref{quad-mod-Z} in the general case. However, to state our results, the important thing is that the fixed coset we consider in that case is an element of $SO(n,m)(\Z)/\Omega(n,m)(\Z)$ (an abelian quotient; see~\cite[7.2.21]{HOM}).
 \par
 In the sequel, we emphasize the case where $(M,Q)$ is a \emph{free hyperbolic module} over $\Z$ (see~\cite[page 197]{HOM}), i.e. $M$ is a $\Z$-module of rank $2n$ equipped with a quadratic form $Q$ (with attached bilinear form denoted $h$) such that there exists a basis of isotropic vectors $\mt{X}=(x_1,\ldots,x_{2n})$ such that,
  $$
 {\rm Mat}_{\mt{X}}h= \Bigl(\begin{array}{cc}
  0 & {\rm Id}\\ {\rm Id} & 0 
  \end{array} \Bigr)\,,
  $$  
where the inner blocks are $n\times n$ matrices.  
\par
 Seen over $\R$, such a quadratic form has signature $(n,n)$ and we will restrict ourselves to such quadratic forms to state Theorem~\ref{exp_decrease} (which is a sample of Theorem~\ref{exp_decrease_gen} in which the case of more general quadratic modules is handled).
 
 \par \medskip
 The question of the irreducibility of the characteristic polynomial of a an element chosen ``at random'', in one of the two types of groups we have just described, is only relevant if no trivial factorisation pattern is imposed by the definition of the groups involved. If we only suppose that $g\in GL(n,A)$, there is no a priori imposed factor for the characteristic polynomial $P_g(T)=\det(T-g)$, but things are different if $g$ is an orthogonal matrix. Indeed, $P_g$ verifies in that case the functional equation:
 \begin{equation} \label{eqfonc}
 P_g(T)=\det(-g)T^NP_g(\frac{1}{T})\,,
 \end{equation}
 where $g$ is assumed to be a $N\times N$ matrix.
 
 


 It seems natural now to wonder about the factorisation properties of the \emph{reduced} characteristic polynomial
which is defined by
$$
 \det(T-g)_{red}=\begin{cases} &\det(T-g)/(T-\det(g))\,,\,\text{if}\,\,N\,\text{is odd}\,,\\
                                &\det(T-g)/(T^2-1)\,,\,\text{if}\,\,N\,\text{is even and}\,\,\det(g)=-1\,,\\
                                 &\det(T-g)\,,\,\text{otherwise}\,.
                 \end{cases}
 $$
 
 Here, the matrix $g$ will always lie in the special orthogonal group attached to $Q$, so that, in the case where $N$ is even, we will always have $\det(T-g)_{red}=\det(T-g)$. Notice moreover that the degree $N_{red}$ of $\det(T-g)_{red}$ is always \emph{even}.

 \par
 \medskip
 Now, with the above notation, let $G$ be the group $GL(n,A)$, for $n\geqslant 2$ (resp. $SO(n,n)(\Z)$, for $n\geqslant 3$), and $G^g$ the normal subgroup $SL(n,A)$ (resp. $\Omega(n,n)(\Z)$) of $G$. Let $S$ be a symmetric generating system for $G^g$ (i.e. for any $s\in S$, we have $s^{-1}\in S$). Notice here that we \emph{do not} assume $G^g$ to be finitely generated, so that $S$ could be infinite (but still countable). Let $(p_s)_{s\in S}$ be a sequence of \emph{strictly positive} real numbers indexed by $S$ satisfying $\sum_{s\in S}{p_s}=1$ and $p_s=p_{s^{-1}}$ for any $s\in S$. Finally let $\alpha$ be a fixed element of $G$.
 \par
Suppose a probability space $(\Psi,\Sigma,{\bf P})$ is given and let $(X_k)_k$ be the (left invariant) random walk on the left coset $\alpha G^g$ defined by
 $$
 X_0=\alpha\indent,\indent X_{k+1}=X_k\xi_{k+1}\,,
 $$ 
 where $(\xi_k)_{k\geqslant 1}$ is a sequence of independent uniformly distributed random variables with values in $S$ and law
 $$
 {\bf P}(\xi_k=s)={\bf P}(\xi_k=s^{-1})=p_s\,,
 $$
 for any $s \in S$.
 \par
  Our main result quantifies the speed of rarefaction of ``non-typical'' elements reached by the $k$-th step of the random walk as $k$ grows. In order to state it in a unified way, the reduced polynomial $\det(T-g)_{red}$ denotes, in the first case, nothing but the usual characteristic polynomial $\det(T-g)$; while in the second case, the ring $A$ denotes nothing but the ring $\Z$. A ``weak'' version of our result can be stated as follows:

  \begin{theorem} \label{exp_decrease}
 With notation as above, there exists a $\beta_1>0$ such that for all $k\geqslant 1$, we have
  $$
  {\bf P}(\det(T-X_k)_{red}\in A[T]\, \text{is reducible})\ll \exp(-\beta_1 k)\,,
  $$
  with $\beta_1$ depending only on the underlying algebraic group $\G/\Q$, on the generating set $S$ and on the sequence $(p_s)_s$ (i.e. on the distribution of the $\xi_k$). Moreover the implied constant depends only on $\G$ and the density of $\Pp$ in the set of all rational prime numbers (in the case where $G=GL(n,\Z[1/\Pp]))$.
  \par
 There exists a $\beta_2>0$ such that for all $k\geqslant 1$, we have
  $$
  {\bf P}(\text{an entry of the matrix}\,\, X_k\,\text{is a square in}\, A)\ll \exp(-\beta_2 k)\,,
  $$
  with the same dependency for $\beta_2$ as for $\beta_1$ and the same dependency for the implied constant as in the previous case.
 \end{theorem}

  In the above statement, the \emph{underlying algebraic group} is ${\bf SL}(n)$ if $G=GL(n,\Z[1/\Pp])$ (with $n\geqslant 3$) and ${\bf SO}(n,n)$ if $G=SO(n,n)(\Z)$ (with $n\geqslant 3$).
  \par
  Of course, the second statement of Theorem~\ref{exp_decrease} is only a very special case of the kind of properties that can be investigated using larege sieve techniques. In Section~\ref{applications}, we give a more general statement of the above Theorem in which the common points of the properties that are likely to be successfully studied via our method appear clearly.

  \section{Estimates for the large sieve constants} \label{Lsieveconst}
 
 In this section we will often need to refer to results coming from the large sieve techniques exposed in the appendix. So, before getting into the proof of Theorem~\ref{exp_decrease}, the reader might either want to check the appendix or assume Propositions~\ref{pr-coset-sieve} and~\ref{pr-group-sieve} (which are self-contained statements) to be true and postpone the reading of the whole appendix.

 \par
 With notation as in Propositions~\ref{pr-coset-sieve} and~\ref{pr-group-sieve}, let $\Lambda$ be a set consisting of odd primes (with the additional condition $\Lambda\cap\mt{P}=\emptyset$ in the case where $G=GL(n,\Z[1/\mt{P}])$) and $\Lstar$ the finite subset of $\Lambda$ consisting of the elements smaller than a given integer $L\geqslant 1$.

  In both Propositions~\ref{pr-coset-sieve} and~\ref{pr-group-sieve} (note that it is natural to emphasize the case where the sets $\Theta_\ell$ are conjugacy invariant as the estimates are improved when using this special property), the heart of the large sieve method lies in the following inequality
   \begin{equation} \label{inegfond}
   {\bf P}(\rho_\ell(X_k)\not\in\Theta_\ell\,\,\text{for all}\, \ell\leqslant L)\leqslant \Delta(X_k,L)H^{-1}\,,
   \end{equation}
   where 
   $$
   H=\sum_{\stackrel{\ell\leqslant L}{\ell\in \Lambda}}{|\Theta_\ell|(|G_\ell^g|-|\Theta_\ell|)^{-1}}
   $$
is the saving factor which depends only on the sieving sets $\Theta_\ell$ and where we denote $\Delta(X_k,L)$ for the constant $\Delta$ of Propostions~\ref{pr-coset-sieve} and~\ref{pr-group-sieve} with the above choice of $\Lstar$. In this section, though, we focus on the large sieve constant $\Delta(X_k,L)$ for which we give an upper bound in the case where $G$ is one of the two groups that Theorem~\ref{exp_decrease} deals with.

\par
  The possibility to obtain the sort of quantitative information stated in Theorem~\ref{exp_decrease} for the random walk $(X_k)$ defined in the introduction depends crucially on the sharpness of the upper bound we can find for the large sieve constants involved. It is not realistic to hope for any useful explicit bound without any assumption on the group $G$ we are working with. As the sums~(\ref{eq-wpitau1}) and~(\ref{eq-wpitau2}) involve representations of the group $G$ (that factor through finite groups), the fact that Lubtotzky's Property $(\tau)$ comes into play is not so surprising. Let us first review some definitons and facts concerning that property.
  
  \subsection{Lubotzky's Property $(\tau)$}
  
  We first recall the definition of the property, as stated in~\cite[1.4]{LZ} or~\cite[Page 49]{Lu}.   
  \begin{definition}
   Let $G$ be a topological group and $\mt{N}=\{N_i\}_i$ a family of normal finite index subgroups of $G$, indexed by a set $I$. The group $G$ has \emph{Property $(\tau)$ with respect to $\mt{N}$} if there exists a finite set $S$ and $\eps>0$ such that, for any unitary irreducible continuous representation $\rho$ of $G$  on a Hilbert space $\mt{H}$ such that $\ker(\rho)\supset N_i$ for some $i$ and that leaves no nonzero vector invariant, we have
   $$
   \max_{s\in S}{\|\rho(s)v-v\|}>\eps\|v\|\,,
   $$
  for all nonzero $v\in\mt{H}$. The pair $(\eps,S)$ is called a $(\tau)$-constant for $G$ and $S$ is called a $(\tau)$-set for $G$.
  \end{definition}

  If, in the definition, we do not require $\rho$ to factor through a subgroup taken from a fixed family, then the group $G$ is said to have Kazhdan's Property $(T)$. So, as is obvious from the definition, Lubotzky's Property is a ``weak'' version of Kazhdan's. That means, of course, that any group with Property $(T)$ also has Property $(\tau)$ with respect to any family of its finite index subgroups. However, that property is indeed strictly weaker: for instance $SL(2,\Z)$ does not have $(T)$ (see~\cite[Prop. 6 page 34]{HV}) but has $(\tau)$ with respect to the family of its congruence subgroups
  $$
  \Bigl(\Gamma(d)=\ker (\rho_d: SL(2,\Z)\ra SL(2,\Z/d\Z))\Bigr)_{d\geqslant 1}\,.
  $$
  
  That last fact, though, does not come for free, as it requires Selberg's result on the eigenvalues of the hyperbolic laplacian acting on $L^2(\Gamma(d)\backslash{\bf H})$ (see, for instance~\cite[4.4]{Lu}).
  
  \begin{remark} 
A useful interpretation  for Property $(\tau)$ is the following: if $G$ is a group and $S$ is a subset of $G$, recall that the \emph{Cayley graph} $\mt{C}(G,S)$ is the oriented graph with vertex set equal to the set of elements in $G$ and where $x$ is connected to $y$ if there exists an $s\in S$ such that $xs=y$. If we suppose that $S$ is symmetric, then $\mt{C}(G,S)$ can be considered as a non oriented graph and if $S$ spans $G$, that graph is connected. Now, with these two additional assumptions, the fact that $G$ has Property $(\tau)$ with resepct to a family $\mt{N}=\{N_i\}_i$ of subgroups is equivalent to the property of \emph{expansion} of the family of Cayley graphs $(\mt{C}(G/N_i,S_i))_i$, where, for each $i$, $S_i$ is the projection of $S$ to the corresponding quotient. More generally a graph $X=(V,E)$ is said to be a $\delta$-expander graph (where $\delta>0$), if, for any subset $A$ of $V$ containing less than one half of the elements of $V$, the number of vertices in $V\setminus A$ which are neighbors of elements of $A$ is at least $\delta|A|$. Moreover, the expansion ration $\delta$ is explicitely related to the $(\tau)$-constant for $G$ with respect to $\mt{N}$. The notion of expander, born in the $1970$'s in order to solve problems linked to networks, has motivated lots of mathematical researches. The beautiful constructions of such families that can be found in~\cite{C},~\cite{M} or~\cite{LPS}, rely heavily on deep mathematical tools. 
 \end{remark}
 
  If $G$ is finitely generated (which is the case in the applications developped by Kowalski in~\cite[Chap. 7]{KoLS}) and has Property $(\tau)$ with respect to $\mt{N}$, it can be shown that any generating set $S$ can be chosen as the $(\tau)$-set for $G$ (see~\cite[Prop. 1.2]{LZ} or~\cite[Th. 4.3.2]{Lu}). For the applications we have in mind, however, we need to work with groups which are not necessarily finitely generated (this is obviously the case for $SL(n,\Z[1/\mt{P}])$ if $\mt{P}$ is infinite). The following result, explained by M. Burger, shows that in this case, we can also choose a $(\tau)$-set among the elements of a generating system for $G$.
  
   \begin{proposition}  \label{taunontf}
  Let $G$ be a group with Property $(\tau)$ with respect to a family of finite index subgroups $\mt{N}=\{N_i\}_i$. Let $S$ be a generating system for $G$, then there exists a finite subset $S_0$ of $S$ which is a $(\tau)$-set for $G$.
  \end{proposition}

     \begin{proof}
  Let $F$ be a (finite) $(\tau)$-set for $G$ and $\delta>0$ such that $(\delta,F)$ is a $(\tau)$-constant for $G$. As $F$ is finite, there exists a subset $S_0$ of $S$ and an integer $n\geqslant 1$ such that $F\subset S_0^n$ (i.e. each element in $F$ can be written as the product of at most $n$ elements of $S_0$). Now let $\pi: G\ra U(\mt{H})$ be a continuous unitary representation of $G$ which factors through $N_i$ for some $i$ (i.e. $\ker \pi\supset N_i$) and without any nonzero invariant vectors in $\mt{H}$. If $v\in \mt{H}$ has norm $1$, then there exists $f=s_0^1\cdots s_0^n\in F$ such that
 \begin{align*}
\delta &\leqslant \left\|\pi(f)(v)-v\right\|\\
       &\leqslant  \left\|\pi(s_0^1\cdots s_0^n)(v)-v\right\|\,.
 \end{align*}     
  
  Using the fact that the representation $\pi$ is unitary, the right hand side of the above inequality can be written
  \begin{align*}
  \left\|\sum_{j=0}^{n-1}{\Bigl(\pi(s_0^1\cdots s_0^{j+1})(v)-\pi(s_0^1\cdots s_0^{j})v\Bigr)}\right\|
  &\leqslant \sum_{j=0}^{n-1}{\left\|\pi(s_0^1\cdots s_0^{j})(\pi(s_0^{j+1})(v)-v)\right\|}\\
  &\leqslant \sum_{j=0}^{n-1}{\left\|\pi(s_0^{j+1})(v)-v\right\|}\,.
  \end{align*}
  
  Combining these last two series of inequalities, we deduce there exists a $t_0\in S_0$ such that 
  $$
  \frac{\delta}{n}\leqslant \left\|\pi(t_0)(v)-v\right\|\,.
  $$
 \end{proof}

  Before explaining how Property $(\tau)$ yields the kind of upper bound we need for the large sieve constants, let us first give (an infinite family of) examples of groups having Property $(\tau)$ (with respect to a certain family of subgroups for each of these examples). Note moreover that in the case where $n=2$ the following lemma provides us with infinitely many examples of groups having Property ($\tau$) (with respect to a suitably chosen family of subgroups) without being Kazhdan groups. These groups, of course, are directly involved in the proof of Theorem~\ref{exp_decrease}.
  
 \begin{lemme} \label{SL-tau}
 Let $\mt{P}$ be a proper subset of the rational primes. For any $n\geqslant 2$ the group $SL(n,\Z[1/\mt{P}])$ has Property $(\tau)$ with respect to the family of its congruence subgroups
 $$
 \Bigl(\ker \pi_d: SL(n,\Z[1/\mt{P}])\ra SL(n,\Z[1/\mt{P}]/d\Z[1/\mt{P}])\Bigr)_{\{d\geqslant 1\mid \,p\nmid d\text{ if } p\in\mt{P}\}}\,.
 $$
 \end{lemme}


 \begin{proof}
 Let $S_1$ be a finite generating system for $SL(n,\Z)$ (the elementary transformations for instance). As already mentioned, $S_1$ is a $(\tau)$-set for $SL(n,\Z)$. The natural inclusion
 $$
 SL(n,\Z)\hookrightarrow SL(n,\Z[1/\mt{P}])
 $$
 enables us to consider a generating set $S\supset S_1$ for $SL(n,\Z[1/\mt{P}])$. Let $m$ be an integer without prime factors in $\mt{P}$; we consider the projection
 $$
 \pi_m: SL(n,\Z[1/\mt{P}])\ra SL(n,\Z[1/\mt{P}]/m\Z[1/\mt{P}])\simeq SL(n,\Z/m\Z).
 $$
 
  The restriction of the morphism $\pi_m$ to $SL(n,\Z)$ is surjective, so, since the family of Cayley graphs $\Bigl(\mt{C}(SL(n,\Z)/(\ker \pi_m\cap SL(n,\Z)), \pi_m(S_1))\Bigr)_m$ (indexed by the integers $m$ coprime to any element in $\mt{P}$) is an expander family (see the discussion preceding remark above and recall that, for $n\geqslant 3$, the group $SL(n,\Z)$ is a Kazhdan group), then so is the family 
  $$
  \Bigl(\mt{C}(SL(n,\Z[1/\mt{P}])/\ker \pi_m, \pi_m(S_1))\Bigr)_m\,.
  $$
  
   A fortiori, the family $\Bigl(\mt{C}(SL(n,\Z[1/\mt{P}])/\ker \pi_m, \pi_m(S))\Bigr)_m$ forms of family of expanders. In other words, the group $SL(n,\Z[1/\mt{P}])$ has Property $(\tau)$ with respect to the family of its congruence subgroups.
  \end{proof}

  \subsection{Upper bounds for $\Delta(X_k,L)$} \label{upper-bounds-delta}
  Coming back to our sieve framework (see the appendix), we now state the key proposition making precise all the assumptions that will need to be verified for the arithmetic groups we consider, in order to obtain a sufficiently sharp bound for the large sieve constants $\Delta$ and $H$  of Propositions~\ref{pr-coset-sieve} and~\ref{pr-group-sieve} (this is the analogue of~\cite[Prop 7.2]{KoLS}).

  \begin{proposition} \label{keyprop}
   We recall $G$ is a discrete group, $G^g$ a normal subgroup of $G$ with abelian quotient $G/G^g$ ant let $T$ be a finite subset of $G/G^g$ in which we let $\alpha$ vary. For a fixed symmetric generating system $S$ of $G^g$ we consider the random walk ($X_k$), defined in the introduction, on the coset $\alpha G^g$ (with $\alpha\in T$) and we suppose
  \begin{itemize}
  \item there exists a relation of \emph{odd} length $c$ among the elements of $S$: 
$$
s_1\cdots s_c=1\,,
$$
  \item the steps $\xi_k$, for $k\geqslant 1$, are independent and independent of $X_0$,
  \item $G^g$ has Property $(\tau)$ with respect to a family $(N_i)_{i\in I}$ of finite index subgroups,
  \end{itemize}
  then there exists $\eta>0$ such that, for any finite dimensional representation 
  $$
  \pi: G\ra GL(V)\,,
  $$
satisfying $\ker \pi_{|G^g}\supset N_i$, for some $i$ and without any nonzero $G^g$-invariant vector, the inequality 
  $$
  |{\bf E}(\langle\pi(X_k)e;f\rangle)|\leqslant||e||\,||f||\exp(-\eta k)\,,
  $$  
holds for all vectors $e$, $f$ in $V$ and all $k\geqslant 0$; $\langle\,;\,\rangle$ denoting a $G$-invariant inner product on $V$. 
  \par
   The constant $\eta$ only depends on the $(\tau)$-constant associated to $(G^g,S,(N_i))$, on the distribution of the $\xi_k$ and on the length $c$ of a fixed relation in $S$. 
  \end{proposition}

 \begin{proof}
  The proof is quite similar to that of~\cite[Prop. 7.2]{KoLS}; many technical points need to be modified, though, so we give the full detail of the arguments here.
\par
\medskip
 We fix an index $i\in I$ and a representation
$$
\pi: G\ra GL(V)\,,
$$
such that the restriction $\pi_{|G^g}$ factors through $N_i$ and has no nonzero $G^g$-invariant vector. Consider
$$
M={\bf E}(\pi(\xi_k))=\sum_{s\in S}{p(s)\pi(s)}\,;
$$
$M$ is a well-defined element of $\End (V)$ since the series defining $M$ converges absolutely (because $\pi$ is a unitary representation and $\sum_s{p(s)}=1$). From $M$, we can then define two other elements of $\End (V)$:
$$
M^+={\rm Id}-M\,, \indent M^-={\rm Id}+M\,.
$$

 Note that these formul\ae\, define two operators which are both independent of $k$ and self adjoint. Indeed, the set $S$ as well as the distribution of the $\xi_k$ are symmetric; moreover the mapping associating its adjoint to an operator is linear and continuous. We also need to define
$$
N_0={\bf E}(\pi(X_0))=\sum_{t\in T}{{\bf P}(X_0=t)\pi(t)}\in \End (V)\,.
$$

The random variables $X_0$ and $\xi_k$ being independent, for $k\geqslant 1$, we have
$$
{\bf E}(\pi(X_k))=N_0 M^k\,,
$$
thus, by linearity,
$$
{\bf E}(\langle\pi(X_k)e;f\rangle)=\langle M^k e;N_0^*f\rangle\,,
$$ 
where $N_0^*$ denotes the adjoint of $N_0$.

\par
\medskip
 As $\sum_{s\in S}{p(s)}=1$ and since for every $s\in S$, $\pi(s)$ is a unitary operator, the eigenvalues of $M$ are in the interval $[-1;1]$. Now, let $\rho$ be the spectral radius of $M$:
 $$
 \rho=\max\{|\gamma|\mid \gamma\, \text{is an eigenvalue of}\, M \}\,.
 $$
 
 We have the inequality
$$
|\langle M^k e;N_0^*f\rangle|\leqslant \left\| e\right\| \left\| f\right\| \rho^k\,,
$$
since the norm of $N_0$ is smaller than $1$.
\par
We need to exhibit a $\delta>0$ independent of $i$ and $\pi$ such that $0\leqslant \rho\leqslant 1-\delta$. We will then be able to choose $\eta=-\log(1-\delta)>0$ for the constant we are looking for. We use the fact that $\rho=\max(\rho^+,\rho^-)$ where $\rho^+$ (resp. $\rho^-$) is the real number equal to the greatest positive eigenvalue of $M$ (resp. equal to the opposite of the smallest negative eigenvalue of $M$). It is enough to prove that $\rho^{\pm}<1-\delta_{\pm}$ for some constants $\delta_{\pm}$ which are independent of $i$ and $\pi$. To that purpose, we use the variational interpretation for the eigenvalues of a self adjoint operator on a finite dimensional Hilbert space. Indeed $1-\rho^+$ (resp. $1+\rho^-$), which is the smallest eigenvalue of $M^+$ (resp. of $M^-$), is equal to 
$$
\lambda=\min_{v\not =0}\frac{\langle Tv;v\rangle}{\left\|v\right\|^2}\,,
$$
where $T=M^{\pm}$. Now applying Proposition~\ref{taunontf}, we know that there exists, for the group $G^g$, a finite $(\tau)$-set $S_0$ included in $S$; this yields
  \begin{align*}
  \frac{\langle M^+v;v\rangle}{||v||^2}&=\frac{1}{2}\sum_{s\in S}{p_s\frac{||\pi(s)v-v||^2}{||v||^2}}\\
&\geqslant \frac{1}{2}\sum_{s\in S_0}{p_s\frac{||\pi(s)v-v||^2}{||v||^2}}\\
&\geqslant\frac{{p_0}^+}{2}\inf_{\varpi}\inf_{v\not=0}\max_{s\in S_0}\frac{||\varpi(s)v-v||^2}{||v||^2}\,,
  \end{align*}
  where $p_0^+=\min_{s\in S_0}{p(s)}>0$ and $\varpi$ runs over the respresentations of $G^g$ without any nonzero $G^g$-invariant vector and which factorize through $N_j$ for some index $j$. Let $\kappa>0$ be such that $(\kappa,S_0)$ is a ($\tau$)-constant for $G^g$ with respect to $(N_i)_{i\in I}$, we can choose 
$$
\delta^+=\frac{\kappa p_0^+}{2}\,.
$$ 

 To determine $\delta^-$, the argument is very close to that of~\cite[Prop. 7.2]{KoLS}. As there exists, by assumption, a relation of odd length $c$ among the elements of $S$, we write, for $v\in V$,
$$
v=\frac{1}{2}\Bigl((v+\pi(s_1)v)-(\pi(s_1)v+\pi(s_1s_2)v)+\cdots+(\pi(s_1\cdots s_{c-1})v+\pi(1)v)\Bigr)\,.
$$
 
 Then, invoking Cauchy-Schwarz's inequality and using the $G$-invariance of the inner product,
$$
\left|v\right\|^2\leqslant\frac{c}{4}\sum_{i=0}^{c-1}{\left\|\pi(r_i)v+\pi(r_is_{i+1})v\right\|^2}\leqslant \frac{c}{4}\sum_{i=0}^{c-1}{\left\|v+\pi(s_{i+1})v\right\|^2}\,,
$$
where $r_0=1$ and $r_i=s_1\cdots s_i$ for $i\geqslant 1$. In particular, we deduce 
$$
\left\|v\right\|^2\leqslant \frac{c}{4}\min_{1\leqslant i\leqslant c}\frac{1}{p(s_i)}\sum_{i=0}^{c-1}{p(s_{i+1})\left\|v+\pi(s_{i+1})v\right\|^2}\,,
$$
then, taking into account the possible repetitions of generators in the sequence $(s_1,\ldots, s_c)$,
$$
\left\|v\right\|^2\leqslant \frac{c^2}{4}\frac{1}{\min{\{p(s_i)\mid 1\leqslant i\leqslant c\}}}\sum_{s\in S}{p(s)\left\|\pi(s)v+v\right\|^2}\leqslant \frac{c^2}{2}\Bigl(\min{\{p(s_i)\mid 1\leqslant i\leqslant c\}}\Bigr)^{-1}\langle M^-v;v\rangle\,.
$$

 Therefore we can choose $\delta^-=\frac{2}{c^2}\min{\{p(s_i)\mid 1\leqslant i\leqslant c\}}>0$.
 \end{proof}

 The next two propositions give, under the assumptions of Proposition~\ref{keyprop} together with an additional hypothesis of \emph{linear disjointness} (which really is a property of ``independence of $\ell$ of the setting'') the upper bound we need for the two large sieve constants $\Delta$ we are working with.

 
 
 \par
  To begin with, we consider the case of the conjugacy coset sieve which is somewhat simpler to handle.

  \begin{proposition} \label{borneconstconj}
  Let $(Y,\Lambda,(\rho_\ell: Y\ra Y_\ell))$ be the conjugacy coset sieve of the appendix (see Proposition~\ref{pr-coset-sieve}). We suppose that
   
   \begin{itemize}
 \item the assumptions of Proposition~\ref{keyprop} are verified,
 \item the system $(\rho_\ell)_{\ell\in\Lambda}$ is \emph{linearly disjoint}, i.e. the restricted product map defined for $\ell,\ell'\in\Lambda$, $\ell\not=\ell'$, by
 $$
 \rho_{\ell,\ell'}=\rho_\ell\times\rho_{\ell'}: G^g\ra G_{\ell,\ell'}^g=G_\ell^g\times G_{\ell'}^g
 $$ 
 is surjective.
 \end{itemize}
  Then, with notation as in Proposition~\ref{pr-coset-sieve}, there exists $\eta>0$ such that
   $$
  \Delta(X_k,L)\leqslant1+L^A\exp(-\eta k)\,,
  $$
 where $\eta>0$ depends only on $G$, $S$ and the distribution of the $\xi_k$ and $A=(3d+2)/2$, with $d=n^2-1$ in the case where $G=SL(n,\Z[1/\mt{P}])$ and $d=(n+m)(n+m-1)/2$ if $G=SO(n,m)(\Z)$.
  \end{proposition}

  \begin{proof}
  From Proposition~\ref{pr-coset-sieve}, we have
   $$ \Delta(X_k,L)\leqslant\max_{\ell\in\Lstar}\max_{\pi\in\Pi_\ell^*}\sum_{\ell'\in\Lstar}\sum_{\tau\in\Pi_{\ell'}^*}{|W(\pi,\tau)|}\,.
 $$
 
  For $\ell, \ell'\in\Lambda$, we need to give an upper bound for the sums
  $$
  W(\varphi_\pi,\varphi_\tau)=\frac{1}{\sqrt{|\hat{\Gamma}_\ell^\pi||\hat{\Gamma}_{\ell'}^\tau|}}{\bf E}(\mathrm{Tr}[\pi,\bar{\tau}]\rho_{\ell,\ell'}(X_k))\,,
  $$
  once rewritten using the trick explained in the appendix after Lemma~\ref{lm-ortho}.
  \par
  With notation as in the appendix (see the paragraphs before and after Lemma~\ref{lm-ortho}), if $[\pi,\bar{\tau}]\rho_{\ell,\ell'}$ has no $G^g$-invariant vector then Proposition~\ref{keyprop} yields an upper bound for
  $$
  |{\bf E}(\mathrm{Tr}[\pi,\bar{\tau}]\rho_{\ell,\ell'}(X_k))|\,.
  $$
  
  Indeed, it is enough to choose $e=f$ running over an orthonormal basis of the representation space $V$ of $[\pi,\bar{\tau}]\rho_{\ell,\ell'}$ and then to sum the terms obtained over $e$.
  \par
   We are reduced to computing the multiplicity of the trivial representation of $G^g$  in the restriction of $[\pi,\bar{\tau}]\rho_{\ell,\ell'}$ to $G^g$. As the sieve setting we work with is supposed to be linearly disjoint, that quantity is the same as the multiplicity of the trivial representation of $G_{\ell,\ell'}^g$ in $[\pi,\bar{\tau}]_{|G_{\ell,\ell'}^g}$. From Lemma~\ref{lm-ortho} of the appendix, we know that multiplicity is zero unless $(\ell,\pi)=(\ell',\tau)$ in which case its value is $|\hat{\Gamma}_\ell^{\pi}|$. Thus, denoting $[\pi,\bar{\tau}]_0$ the part of $[\pi,\bar{\tau}]$ without any nonzero $G^g$-invariant vector, we deduce:
   $$ \mathrm{Tr}[\pi,\bar{\tau}]\rho_{\ell,\ell'}(X_k)=|\hat{\Gamma}_\ell^\pi|\delta((\ell,\pi),(\ell',\tau))+\mathrm{Tr}[\pi,\bar{\tau}]_0\rho_{\ell,\ell'}(X_k)\,.
   $$
   
Applying Proposition~\ref{keyprop} to the representation $[\pi,\bar{\tau}]_0\rho_{\ell,\ell'}$ of $G^g$ yields
   $$
   |W(\pi,\tau)-\delta((\ell,\pi),(\ell',\tau))|\leqslant (\dim \pi)(\dim\tau)\exp(-\eta k)\,,
   $$
  where we use the trivial upper bound $\Bigl(\sqrt{|\hat{\Gamma}_\ell^\pi||\hat{\Gamma}_{\ell'}^\tau|}\Bigr)^{-1}\leqslant 1$. 
  \par 
  The result follows from exploiting such trivial bounds as:
  $$
  \dim \pi\leqslant \sqrt{|G|}\indent,\indent \sum_{\pi \in \text{irr}(G)}{\dim \pi}\leqslant |G|\,,
  $$
  for any irreducible complex representation $\pi$ of a finite group $G$ (see~\cite[Chap. 5]{KoLS} for better bounds for such quantities).
  
  \end{proof}

  In the next proposition, we give an upper bound for $\Delta(X_k,\Lstar)$ in the case of the non-conjugacy coset sieve, very close to the one stated in Proposition~\ref{borneconstconj}. However, needing to use another equivalence relation to define the orthonormal basis for the $L^2$ space involved (compare the statements of Proposition~\ref{pr-coset-sieve} and~\ref{pr-group-sieve}), the above proof cannot be directly adapted to the case of the non-conjugacy coset sieve. Indeed, to prove the following result, we use the remark following Proposition~\ref{pr-group-sieve} about the generalisation of the sieve statements of the appendix to a framework in which we do not only use primes but more generally squarefree integers to perform the sieve.

 \begin{proposition} \label{exp_sum_maj}
 Let $(Y=\alpha G^g,\Lambda,(\rho_\ell: Y\ra Y_\ell))$ be the non conjugacy coset sieve of the appendix (see, e.g., Proposition~\ref{pr-group-sieve}). For any fixed integer $L\geqslant 1$ and under the same assumptions as in Proposition~\ref{borneconstconj}, there exists $\eta'>0$ such that
 $$
 \Delta(X_k,L)\leqslant 1+L^{A'}\exp(-\eta'k)\,,
 $$
 where $\eta'>0$ depends only on $G$, the $(\tau)$-constant for $G^g$, the distribution of the $\xi_k$ and $A'=(17d+4)/4$ with $d=n^2-1$ if $G=SL(n,\Z[1/\mt{P}])$ and $d=(n+m)(n+m-1)/2$ if $G=SO(n,m)(\Z)$.
 
 \end{proposition}

 \begin{proof}
 As in the proof of the previous proposition, we need to evaluate a sum involving group characters. The point is that the maximal contribution, in those sums, comes from the function corresponding to the trivial representation. Following that idea, we apply an ``equidistribution approach'' in order to obtain the estimate we are aiming at:
  $$
 {\bf E}(\langle[\pi,\bar{\tau}](\rho_{\ell,\ell'}(X_k))\tilde{e}\, ;\tilde{f}\rangle_{[\pi,\bar{\tau}]})=\sum_{y\in Y_{\ell,\ell'}}{\langle[\pi,\bar{\tau}](y)\tilde{e}\, ;\tilde{f}\rangle_{[\pi,\bar{\tau}]}I_y}\,,
 $$
 where the notation are those of~(\ref{Wnonconjsieve}) and where $I_y$, on the right hand side of the equality is defined by:
 $$
 I_y={\bf P}( \rho_{\ell,\ell'}(X_k)=y)\,.
 $$

  To evaluate $I_y$ we decompose the characteristic function $\chi_y$ of $\{y\}$ in Fourier series. To that purpose, we need to extend by multiplicativity the result of Proposition~\ref{pr-group-sieve} to the case of squarefree integers (not only primes, see the remark in the appendix, after Proposition~\ref{pr-group-sieve}). In $L^2(Y_{\ell,\ell'},\nu_{\ell,\ell'})$, the following equality holds:
 $$
 \chi_y=\sum_{\varphi\in \mt{B}_{\ell,\ell'}}{\langle\varphi;\chi_y\rangle\varphi}=\sum_{\varphi\in \mt{B}_{\ell,\ell'}}{\varphi(y)|G_{\ell,\ell'}^g|^{-1}\varphi}\,.
 $$
 
Thus, we obtain
 \begin{align*}
 I_y ={\bf E}(\chi_y(\rho_{\ell,\ell'}X_k))&=\frac{1}{|G_{\ell,\ell'}^g|}\sum_{\varphi\in\mt{B}_{\ell,\ell'}}{\varphi(y)}{\bf E}(\varphi\rho_{\ell,\ell'}(X_k))\\                               &=\frac{1}{|G_{\ell,\ell'}^g|}+\frac{1}{|G_{\ell,\ell'}^g|}\sum_{\varphi\in\mt{B}_{\ell,\ell'}\setminus\{1\}}{\varphi(y){\bf E}(\varphi\rho_{\ell,\ell'}(X_k))}\,.
 \end{align*}
 
 Now, if $\varphi=\varphi_{\pi_{\ell,\ell'},e',f'}$ is an element of $\mt{B}_{\ell,\ell'}$ (up to a suitable normalisation, see Lemma~\ref{basenonconjsieve}) different from $1$, we know in particular that $\pi_{\ell,\ell'}$ is an irreducible representation of $G_{\ell,\ell'}$ and the quantity for which we need to find an upper bound is
$$
|{\bf E}(\langle\pi_{\ell,\ell'}\rho_{\ell,\ell'}(X_k)e' ; f'\rangle_{\pi_{\ell,\ell'}})|\,.
$$ 

 To apply Proposition~\ref{keyprop}, we need to determine the multiplicity of the trivial representation of $G_{\ell,\ell'}^g$ in the restriction of $\pi_{\ell,\ell'}$ to $G_{\ell,\ell'}^g$. As we assume linear disjointness of the sieve setting, this multiplicity is the same as that of  $1_{G^g}$ in $\pi_{\ell,\ell'}\rho_{\ell,\ell'}$ for the group $G^g$. Applying Lemma~\ref{lm-ortho} (once extended by multiplicativity) with $\pi=\pi_{\ell,\ell'}$, $\tau=1_{G_\ell}$ (note that the assertion of Lemma~\ref{lm-ortho} remains valid for the trivial representation, see~\cite[Proof of Lemma $3.2$]{KoLS}), and using the fact that $\ell$ is a prime factor of ${\rm ppcm}(\ell,\ell')$, we see that multiplicity is zero unless ${\rm ppcm}(\ell,\ell')=\ell$ (i.e. $\ell'=\ell$) and ${\pi_{\ell,\ell'}}_{|G_{\ell,\ell'}^g}=1_{|G_{\ell,\ell'}^g}$ (or more precisely, applying~\cite[Lemma 3.2]{KoLS}, $\pi_{\ell,\ell'}\otimes\psi\simeq 1_{G_{\ell,\ell'}}$ for a certain character $\psi$ of $G_{\ell,\ell'}/G_{\ell,\ell'}^g$). In particular, $\pi_{\ell,\ell'}$ has dimension $1$ and, for every vector $e'$ with norm $1$ spanning the representation space of $\pi_{\ell,\ell'}$ and every $g\in G_{\ell,\ell'}^g$,
$$
\langle\pi_{\ell,\ell'}(g)e';e'\rangle=\langle e';e'\rangle=1\,,
$$ 
where the index $\pi_{\ell,\ell'}$ is purposely omitted to avoid the use of too much notation. Thus, with notation as in Lemma~\ref{basenonconjsieve}, we deduce that $\varphi_{\pi_{\ell,\ell'},e',e'}\sim 1$, which is a contradiction.

\par
\medskip
 Invoking proposition~\ref{keyprop} now yields an $\eta'>0$ such that for all $\varphi=\varphi_{\pi_{\ell,\ell'},e',f'}\in\mt{B}_{\ell,\ell'}\setminus\{1\}$, 
$$
|{\bf E}(\langle\pi_{\ell,\ell'}\rho_{\ell,\ell'}(X_k)e';f'\rangle_{\pi_{\ell,\ell'}})|\leqslant \exp(-\eta' k)\,.
$$  

Finally, for the quantity 
$$
(\dim\pi)^{(-1/2)}(\dim\tau)^{(-1/2)}
 \Bigl|W(\varphi_{\pi,e,f},\varphi_{\tau,\epsilon,\phi})-\delta(\varphi_{\pi,e,f},\varphi_{\tau,\epsilon,\phi})\Bigr|\,,
 $$ 
 we obtain the following upper bound (note that the inverse of the denominator of the normalisation factor is trivially smaller than $1$):
$$
\Bigl|\frac{1}{|G_{\ell,\ell'}^g|}\sum_{y\in Y_{\ell,\ell'}}{\Bigl(\langle[\pi,\bar{\tau}](y)\tilde{e};\tilde{f}\rangle_{[\pi,\bar{\tau}]}}\sum_{\varphi\in\mt{B}_{\ell,\ell'}\setminus\{1\}}{\varphi(y){\bf E} (\varphi\rho_{\ell,\ell'}(X_k))\Bigr)} \Bigr|\,.
$$

 Applying Cauchy-Shwarz's inequality, we obtain
 \begin{align*}
|\langle[\pi,\bar{\tau}](y)\tilde{e};\tilde{f}\rangle_{[\pi,\bar{\tau}]}| &\leqslant  ||[\pi,\bar{\tau}](y)\tilde{e}||_{[\pi,\bar{\tau}]}\,||\tilde{f}||_{[\pi,\bar{\tau}]}\\
&\leqslant 1             
\end{align*}
and more generally $|\varphi(y)|\leqslant 1$, for all $y\in Y_{\ell,\ell'}$ and all $\varphi\in\mt{B}_{\ell,\ell'}\setminus\{1\}$. We deduce an upper bound for $(\dim\pi)^{(-1/2)}(\dim\tau)^{(-1/2)}
 \Bigl|W(\varphi_{\pi,e,f},\varphi_{\tau,\epsilon,\phi})-\delta(\varphi_{\pi,e,f},\varphi_{\tau,\epsilon,\phi})\Bigr|$, by using the triangle inequality,
 $$
 (|\mt{B}_{\ell,\ell'}|-1)\exp(-\eta' k)\,.
$$

 Now, by classical group representation theory, 
$$
|\mt{B}_{\ell,\ell'}|\leqslant \sum_{\pi\in {\rm irr}(G_{\ell,\ell'})}{(\dim\pi)^2}=|G_{\ell,\ell'}|.
$$

 We finally deduce an upper bound for the large sieve constant:
 $$
 \Delta(X_k,L)\leqslant 1+L^{A'}\exp(-\eta k)\,,
 $$
 with $A'=(17d+4)/4$ and either $d=n^2-1$ if $G=GL(n,\Z[1/\Pp])$, or $d=(n+m)(+m-1)/2$ if $G=SO(n,m)(\Z)$. Indeed, from the argument above, we just need to use the same kind of trivial bounds as the ones at the end of the proof of Proposition~\ref{borneconstconj} as well as the obvious inequality $|G_{\ell,\ell'}|\leqslant |G_\ell||G_{\ell'}|$.
 
 \end{proof}

 \section{Local densities for polynomials and orthogonal matrices} \label{localdensities}
 In this section, which is independent of the others, we compute different densities in subsets of the ring $\F_\ell[T]$ or of the orthogonal group $O(N,\F_\ell)$ (notice that we do not assume anything here on the integer $N$ and, in particular, we do not distinguish between the split and non split model for the orthogonal group, in the case $N$ is even).
 \par
 The goal of this section is to give enough quantitative information in order to find a useful lower bound for the constant $H$ appearing in~(\ref{inegfond}). However, that section having an interest of its own, we do not restrict ourselves to the computations that are strictly needed for the purpose of the paper. The style in which we expose the different estimates we are interested in is very much inspired by~\cite[Section 3]{Ch}. Doing so, it is easy to point out the common points as well as the differences between the symplectic case (treated by Chavdarov in loc. cit.) and the orthogonal case. We will namely highlight that the lack of good topological properties for the orthogonal group imposes to be very careful in the statement of our results (such precautions need not be taken in the case of the symplectic group).

 \subsection{Review of orthogonal groups over finite fields}
 
 We briefly recall some basic facts and notation about orthogonal groups, as exposed in~\cite[Section 6]{Ka}. The proofs and details can be found e.g. in~\cite{ABS}.
 \par
 Let $V$ be a vector space with dimension greater or equal to $2$ over a fixed finite field $\F_q$ with characteristic different from $2$. We assume we are given a non degenerate quadratic form $Q$ on $V$ (we will denote by $\Phi$ the bilinear form attached to $Q$). If $T(V)$ denotes the tensor algebra associated to $V$, we can consider the ideal $\mathfrak{I}(Q)$ generated by the elements $x\otimes x-Q(x).1$, where $x$ runs over the elements of $V$. The quotient algebra $Cl(V,Q)=T(V)/\mathfrak{I}(Q)$ is the \emph{Clifford algebra} of $V$ with respect to $Q$. That construction yields a natural injection
 $$
 i_Q: V\ra T(V)\ra Cl(V,Q)\,,
 $$
 which enables us to see $Cl(V,Q)$ as the solution  of the following universal problem: for every morphism of $\F_q$-vector spaces $f: V\ra A$, where $A$ is an $\F_q$-algebra satisfying $f(x)^2=Q(x).1_A$, there exists a unique $\F_q$-algebra homomorphism $\tilde{f}: Cl(V,Q)\ra A$ such that $\tilde{f}\circ i_Q=f$.
 \par
 Now, the involution $v\mapsto -v$ in $V$, can be extended to an involution, denoted $I$, of $Cl(V,Q)$. Another morphism plays a crucial role in the theory of Clifford algebras: it is the antiautomorphism $t: Cl(V,Q)\ra Cl(V,Q)$ coming from the natural antiautomorphism defined on the $k$-th tensor power of $V$ by
 $$
 v_1\otimes v_2\otimes\cdots\otimes v_k\mapsto v_k\otimes\cdots\otimes v_2\otimes v_1\,.
 $$
 
  Let $Cl^\times$ be the group of invertible elements of $Cl(V,Q)$. It acts on $Cl(V,Q)$ via the morphism $\rho$ defined by
  $$
  \rho(u)x=I(u)xu^{-1}\,,
  $$
  for every $u\in Cl^\times$ and $x\in Cl(V,Q)$. The elements of $Cl^\times$ that leave $V$ globally invariant form a subgroup, denoted $C^\times$ and called \emph{the Clifford group}, of $Cl^\times$. Typical elements of $C^\times$ are the images in $Cl(V,Q)$ of non isotropic vectors $v\in V$, since the transformation $x\mapsto I(u)xu^{-1}$ is then the reflexion with respect to the hyperplane which is orthogonal to $v$. In fact, any element of $C^\times$ is a scalar multiple of a product of such vectors (the transformation associated to that element being an automorphism of the quadratic module $(V,Q)$).
  \par
  Finally, we define the map $Norm: u\in C^\times\mapsto t(u)u$ which takes its values in $\F_q^\times$ (see~\cite[Prop. 3.3 and 3.8]{ABS} where the proof, given in the case where the base field is $\R$, can be easily adapted to the finite field case). The \emph{spinor group} $\spin(V,Q)$ is then defined as the subgroup
  $$
  \spin(V,Q)=(\ker Norm)^I\,,
  $$
  of the elements of $C^\times$ that are fixed by $I$. When $V$ is $N$-dimensional and there is no ambiguity on the chosen quadratic form $Q$, we will denote that group $\spin(N,\F_q)$ instead of $\spin(V,Q)$. With such notation, the group $\spin(N,\F_q)$ can in fact be seen as the group of $\F_q$-rational points of an algebraic group defined over $\F_q$, denoted ${\bf Spin}(N)$, and which we will also refer to as the spinor group. It is well-known that the spinor group is a connected simply-connected algebraic group and that it is in fact the universal cover of the special orthogonal group ${\bf SO}(N)/\F_q$. In other words~(\cite[page 189]{Hu}) there exists an isogeny $\varphi$ such that we have an exact sequence of algebraic groups
 \begin{equation} \label{exact_gp_alg}
\begin{CD}
 1@>>> \{\pm 1\} @>>> {\bf Spin}(N) @>\varphi >> {\bf SO}(N)@>>> 1\,.
\end{CD} 
\end{equation}

 Thus the spinor group shares the same dimension as the special orthogonal group $N(N-1)/2$ and the same rank $\lfloor N/2 \rfloor$.

  \begin{remark}
   In all of the above, we do not need to assume that the base field is a finite field. Every construction and definition we have recalled can in fact be stated for quadratic modules over any perfect field.
  \end{remark}
 
 Now, if $\overline{\F_q}$ denotes a fixed separable closure of $\F_q$, the short exact sequence~(\ref{exact_gp_alg}) gives rise to the following exact sequence of $\Gal(\overline{\F_q}/\F_q)$-invariant groups:
 \begin{equation} \label{defnormspin}
\begin{CD}
1 \longrightarrow \{\pm 1\}@>>> \spin(N,\F_q) @>\rho>> SO(N,\F_q) @>\nsp>> \{\pm 1\}\longrightarrow 1\,,
\end{CD}
\end{equation} 
where the group homomorphism $\nsp$ is called the \emph{spinor norm} and can be defined as follows. For a non isotropic vector $v\in V$, the image of the reflection with respect to $v$ by $\nsp$ is the class of $Q(v)$ in $\F_q^\times/(\F_q^\times)^2$ (the group of classes modulo nonzero squares of $\F_q^\times$). The morphism $\nsp$ is entirely determined by the images of those elements and it can be shown (see~\cite[Section 6]{Ka}) that $\nsp$ can be extended to a surjective morphism from $O(N,\F_q)$ onto $\{\pm 1\}$ since we have supposed $N\geqslant 2$.
\par
 Finally, we will denote $\Omega(N,\F_q)$ the image of $\rho$ in~(\ref{defnormspin}) (i.e. the kernel of $\nsp$). That group which is of great importance in our sieving context, is easily seen to be the derived group of both $SO(N,\F_q)$ and $O(N,\F_q)$ (see~\cite[5.17]{Ar.E}). We note that it can be directly defined from $O(N,\F_q)$ by saying that it is the simultaneous kernel of the determinant and the spinor norm. 

\par
\medskip
 To perform the density computations we need, we will have to estimate the cardinalty of sets of orthogonal matrices with fixed determinant and/or spinor norm. In practice, it will be convenient to relate those quantities to the number of polynomials with coefficients in the base field that can be realized as characteristic polynomials of such matrices. To exhibit the link between these cardinalities, the crucial point lies in the possibility to ``see'' the value of the spinor norm and of the determinant of a matrix $g$ in the coefficients of its characteristic polynomial. This is, of course, very easy in the case of the determinant. But as far as the spinor norm is concerned, we do not see any obvious reason for such an explicit link to exist. However the following beautiful result of Zassenhaus (which we recall here, in view of its importance), gives us, under certain conditions, the kind of link we need:
 
 \begin{theorem}[Zassenhaus,1962] \label{th_zass}
 If $g$ is an element of the orthogonal group associated to the quadratic space $(V,Q)$ (satisfying the same assumptions as above), then, provided $-1$ is not an eigenvalue of $g$, we have:
 $$
 \det(\frac{g+1}{2})=\nsp(g)\,,
 $$
 where both $\det$ and $\nsp$ are seen as applications with values in $\{\pm 1\}$.
\end{theorem} 

 In~\cite{Za}, Zassenhaus first defines the spinor norm via the formula of the theorem above (see $(2.1)$ in loc. cit.) and then proves that definition coincides we the one we gave earlier in this section (see the corollary of~\cite[Th. page 446]{Za}).

 \begin{remark}
  Over finite fields of odd characteristic, we know that there are two isomorphism classes of quadratic forms (classified by the value of the discriminant of the quadratic form). In the case where the dimension is odd, these two classes give rise to the same orthogonal group, but this does not hold if the dimension is even. Indeed if $N=2n$, two distinct models for the orthogonal group $O(2n,\F_q)$ need to be distinguished: they are respectively called the \emph{split} and \emph{nonsplit} model, refering to the algebraic group ${\bf O}(2n)$ being split or not over $\F_q$ (see~\cite[Section 6]{Ka} for examples of quadratic forms corresponding to each of these two models). Note that the computations in the sequel are performed independently of the chosen model of orthogonal group. However, we will see later how a result of Baeza makes the choice of the split or of the nonsplit model come back into play.
 \end{remark}

 \subsection{Characteristic polynomials of orthogonal matrices over finite fields}
  Let $\ell$ be an odd prime number and $N\geqslant 2$ an integer. As in the previous subsection, $(V,Q)$ still denotes a quadratic space over $\F_\ell$ and $Q$ is still assumed to be non degenerate. For $g$ an element of $O(N,\F_\ell)$, we denote by $P_g$ the \emph{reversed} characteristic polynomial of $g$:
  $$
  P_g(T)=\det(1-Tg)\,.
  $$
  
 This polynomial also satisfies the functional equation~(\ref{eqfonc}). A short proof of that fact goes as follows: if $\sigma$ is the automorphism of the ambiant quadratic space $(V,Q)$ attached to the matrix $g$, we denote by $\mt{Q}$ the matrix of the quadratic form $Q$ wrtitten in the basis in which the matrix of $\sigma$ equals $g$. Then we have $g\mt{Q}(^tg)=\mt{Q}$; we deduce
 \begin{align*}
 P_g(T)&=\det(1-Tg)=\det(1-T(^tg))\\
       &=\det(1-T(\mt{Q}^{-1}g^{-1}\mt{Q}))\\
       &=T^N\det(g^{-1})\det(g/T-1)=(-T)^N\det(g)P_g(1/T)\,.   
 \end{align*}
  
 Here one should be careful that $P_g$ no longer designates the ``usual'' characteristic polynomial $\det(T-g)$, as in the introduction. However, it is easily seen (and it can be proved in the exact same way), that the reversed characteristic polynomial $P_g$ still satisfies~(\ref{eqfonc}). Moreover, another motivation to change the notation is that Chavdarov works with reversed characteristic polynomials in~\cite{Ch}, so that we can easily understand to what extent his results can be transposed to the orthogonal case. 
 \par
 The functional equation~(\ref{eqfonc}) imposes $P_g$ to be ``almost self-reciprocal'', i.e., its roots $\alpha_1,\ldots,\alpha_N$ can be reordered in such a way that
 $$
 \left\{ \begin{array}{l} \alpha_i\alpha_{N+1-i}=1\,,\, 1\leqslant i\leqslant n,\,\, \text{if} \, N=2n\,,\\
  \alpha_i\alpha_{N+1-i}=1\,,\, 1\leqslant i\leqslant n\,\, \text{and}\,\, \alpha_{n+1}=\det(g),\,\, \text{if} \,\, N=2n+1\,,\end{array}\right.
 $$
 
 That leads naturally to the consideration of the set of polynomials
 $$
 M_{N,\ell}=\Bigl\{1+b_1T+\cdots+b_NT^N\mid b_i\in\F_{\ell}, b_N^2=1 \,\, \text{and}\,\, b_{N-i}=b_Nb_i,\,\, \text{if}\,\,0\leqslant i\leqslant \lfloor N/2\rfloor\Bigr\}\,.
$$

 Before studying certain subsets of $M_{N,\ell}$, we recall a result due to Edawrds which gives in the case where $N$ is even a (quite surprising at first) link between the discriminant of a polynomial $f\in M_{N,\ell}$ and the values $f$ takes at $\pm 1$.
 
  \begin{lemme}[Edwards] \label{lem_edwards}
   Let $N$ be an even integer and let $f\in M_{N,\ell}$ be a monic separable polynomial, then we have
   $$
   \disc(f)\equiv f(1)f(-1)\,(\mathrm{mod}\,(\F_\ell^\times)^2)\,.
   $$
   \end{lemme}
   
   That result is obtained by combinig Theorem $1$ and Theorem $2$ of~\cite{E} (note that the definition of the discriminant of a polynomial used in loc. cit. is not the standard one hence the statement of Lemma~\ref{lem_edwards} only coincides with the one of~\cite{E} up to sign).
 
\par\medskip

 We are interested in the cardinality of certain subsets of $M_{N,\ell}$. Our first result in this direction, which is very close to~\cite[Lemma 3.2]{Ch}, deals with the subset of irreducible polynomials in $M_{N,\ell}$, or rather of those which are irreducible \emph{once reduced}. Indeed, using the notation of the introduction and remembering we consider reversed characteristic polynomials here, we will use the notation
  \begin{itemize}
  \item if $N=2n$,  
  $$
  K_{N,\ell}^{\eps}=\Bigl\{f(T)\in M_{N,\ell}\mid f(T)_{red}=\frac{f(T)}{1-\eps T^2}\,\, \text{is irreducible}\Bigr\}\,,
  $$
  with $\eps=1$ if $b_N=-1$ and $\epsilon=0$ otherwise,
  \item if $N=2n+1$, 
  $$
  K_{N,\ell}^{\eps}=\Bigl\{f(T)\in M_{N,\ell}\mid f(T)_{red}=\frac{f(T)}{1-\eps T}\,\, \text{is irreducible}\Bigr\}\,,
  $$
with  $\epsilon=\pm1$ if $b_N=\mp 1$.
  \end{itemize}
 In both cases we will denote by $N_{red}$ the degree of the reduced polynomial $f_{red}$.
 
 \par
 \medskip
 For those sets we have the following estimates:
 
  \begin{proposition} \label{card_pol_irr}
 \begin{itemize}
 \item If $N$ is even,
 $$
 \frac{1+\ell^{N/2-\eps}}{N-2\eps}\geqslant |K_{N,\ell}^{\eps}|\geqslant \frac{\ell^{N/2-\eps}}{N-2\eps}-\sqrt{\ell^{N/2-\eps}}\,.
 $$
 \item If $N$ is odd,
 $$
 \frac{1+\ell^{(N-1)/2}}{N-1}\geqslant |K_{N,\ell}^{\eps}|\geqslant \frac{\ell^{(N-1)/2}}{N-1}-\sqrt{\ell^{(N-1)/2}}\,.
 $$
 \end{itemize}
 \end{proposition}

 \begin{proof} 
  It is enough to consider the case where $N$ is even and $\eps=0$. Indeed, if $N$ is even and $\eps=1$ (which means that the leading coefficient of $f$ is $-1$), we are interested in the irreducibility of $g(T)=f(T)/(1-T^2)$. The leading coefficient of $g$ is obviously $1$ and it is clear that $g\in M_{N-2,\ell}$. Thus $g\in K_{N-2,\ell}^0$ as soon as $f\in K_{N,\ell}^1$.   
 \par
 In the case where $N$ is odd, then $g(T)=f(T)/(1-\epsilon T)$ has leading coefficient $1$ and $g\in M_{N-1,\ell}$ so that $g\in K_{N-1,\ell}^0$ as soon as $f\in K_{N,\ell}^\eps$. In other words:
 \begin{itemize}
 \item if $N$ is even, $|K_{N,\ell}^{1}|=|K_{N-2,\ell}^{0}|$, and
 \item if $N$ is odd, $|K_{N,\ell}^{\eps}|=|K_{N-1,\ell}^{0}|$. 
 \end{itemize}
 
 We are now reduced to computing the number of elements of 
  $$
  K_{N,\ell}^{0}=\Bigl\{f(T)\in M_{N,\ell}\mid N\,\,\text{even}\,,\,b_{N-i}=b_i,\, 0\leqslant i\leqslant N/2\,,\,f\,\,\text{monic, irreducible}\Bigr\}\,.
  $$

 This happens to be exactly the set $K_g^{1}$ studied by Chavdarov in~\cite[Lemma 3.2]{Ch} for $g=N/2$. In loc. cit., Chavdarov proves that for $N$ even,
 $$
 \frac{1+\ell^{N/2}}{N}\geqslant |K_{N,\ell}^{0}|\geqslant \frac{\ell^{N/2}}{N}-\ell^{N/4}\,,
 $$      
 and this completes the proof.
\end{proof}

  Restricting ourselves to polynomials $f$ satisfying $f=f_{red}$ (i.e. $N=\deg f$ is even and $f$ is monic), we are interested in the same kind of computation as above with the extra condition that $f$ must be such that $f(1)$ and $f(-1)$ are in a given class (not necessarily the same for $f(1)$ and $f(-1)$) of $\F_\ell^\times$ modulo its subgroup of nonzero squares, i.e., by Lemma~\ref{lem_edwards} and Theorem~\ref{th_zass}, we work with polynomials having fixed discriminant and that can only be characteristic polynomials for elements with fixed spinor norm. At first, it seems likely that we will get a ``good'' proportion of such polynomials among the set of irreducible self-recipocal polynomials of even degree (that proportion should roughly be $1/4$). However, the following result of Meyn (see~\cite[Th. 8]{meyn}) tells us that this intuition is wrong:
  
  \begin{proposition}[Meyn]\label{prop-meyn}
  Let $f$ be a self-reciprocal monic polynomial of even degree $N$ over $\F_\ell$. Let us write
  $$
  f=x^{N/2}h(x+x^{-1})\,
  $$
  with $h$ monic of degree $N/2$. If $h$ is an irreducible polynomial then $f$ is irreducible iff $h(2)h(-2)$ is a \emph{nonsquare} of $\F_\ell$.
  \end{proposition}
 
 Moreover, with notation as in the above satement, it is easy to see that if we start with an irreducible $f$, the attached polynomial $h$ will also be irreducible. So any self-reciprocal monic polynomial $f$ with $f(1)$ and $f(-1)$ chosen in such a way that $(-1)^{N/2}f(1)f(-1)$ is a square in $\F_\ell$ is \emph{not} irreducible. This means that out of the four classes determined by the imposed value in $1$ and $-1$ modulo squares, only two are non empty. 
 \par 
 The following result asserts that the expected equidistribution property holds for the two non empty classes.

  \begin{lemme} \label{adaptcarlitz}
   Let $N\geqslant 4$ be an even integer. Then, if $\eps_\ell^{(1)}, \eps_\ell^{(2)}$ are (non necessarily distinct) fixed elements of a fixed set $\{1,\eps_0\}$ of representatives of $\F_\ell^\times/(\F_\ell^\times)^2$ and if $(-1)^{N/2}\eps_\ell^{(1)}\eps_\ell^{(2)}$ is a nonsquare of $\F_\ell$, then we have 
   $$
   \Bigl|\Bigl\{f\in M_{N,\ell}\mid f\,\text{is irreducible and}\,\,f(a)\equiv \eps_\ell^1\,,f(b)\equiv \eps_\ell^2\Bigr\}\Bigr|\geqslant  \frac{\ell^{N/2}}{2N}\Bigl(1-\frac{2(1+N)}{\ell}\Bigr)\,.
   $$
   where, in the set of the left hand side, congruences are taken modulo the group of nonzero squares of $\F_\ell$.
  \end{lemme}
  
  \begin{proof}
  By Meyn's Theorem and the discussion preceding Lemma~\ref{adaptcarlitz}, the set we are interested in is in one-to-one correspondence with
  $$
 \Bigl\{h\in \F_\ell[T]\text{ monic of degree } N/2\mid h\,\text{is irreducible and}\,\,h(2)\equiv \eps_\ell^{(1)}\,,h(-2)\equiv (-1)^{N/2}\eps_\ell^{(2)}\Bigr\}\,.
  $$
  
  If we vary $\eps_\ell^{(1)}$ and $\eps_\ell^{(2)}$ in $\{1,\eps_0\}$, we see that computing the cardinality of the above set amounts to evaluating the four following character sums
  $$
  \frac{1}{4}\sum_h{(1\pm\chi_\ell(h(2)))(1\pm\chi_\ell(h(-2)))}\,,
  $$
  where the sum is taken over all monic degree $N/2$ irreducible polynomials of $\F_\ell[T]$ and where $\chi_\ell$ denotes the Legendre character of $\F_\ell$.
  \par
  
  It is enough to focus on the study of the sum
  $$
 \mt{S}=\sum_h{(1+\chi_\ell(h(2)))(1+\chi_\ell(h(-2)))}=\sum_{h}{1}+\sum_h{\chi_\ell(h(2))}+\sum_h{\chi_\ell(h(-2))}+\sum_h{\chi_\ell(h(2)h(-2))}\,.
  $$  
  
   The first sum of the right hand side is nothing but the number of irreducible monic polynomials of degree $N/2$ in $\F_\ell[T]$. There are well known lower bounds for that quantity. For our purpose, it is enough to use the inequality (see, e.g.~\cite[Lemma 3.1]{Ch})
   $$
   \sum_h{1}\geqslant \frac{2\ell^{N/2}}{N}-\ell^{N/4}\,.
   $$ 
  
  Next we consider both the sums $\sum_h{\chi_\ell(h(2))}$ and $\sum_h{\chi_\ell(h(-2))}$. As $N/2\geqslant 2$ (and thus an irreducible $h$ of degree $N/2$ cannot be $X\pm 2$), we have, from the definition of $\chi_\ell$:
  $$
  \sum_h{\chi_\ell(h(2))}=2\sum_{a \text{ nonzero square }}{|\{h\mid h(2)=a\}|}-\sum_h{1}
  $$
  
  Notice that the summand of the right hand side does not depend on the point of $\F_\ell$ at which $h$ is evaluated (i.e. imposing the value of $h$ at any point of $\F_\ell$ yields a set with the same cardinality as the analogue set where the value imposed is $h(0)$), so using on the one hand the lower bound (see~\cite[Appendix B, formula B.8]{KoLS})
  $$
  |\{h\in\F_\ell[T]\mid h \text{ monic irreducible }, \deg h=N/2, h(2)=a\}|\geqslant \frac{2\ell^{N/2-1}}{N}-\ell^{N/4}\,,
  $$
  and on the other hand the upper bound (see~\cite[Lemma 3.1]{Ch} or~\cite[Appendix B, Lemma B.1]{KoLS})
  $$
|\{h\in\F_\ell[T]\mid h \text{ monic irreducible }, \deg h=N/2\}|\leqslant \frac{2\ell{N/2}}{N}\,,
  $$
  we get:
  \begin{align*}
  -\sum_h{\chi_\ell(h(2))}&\leqslant \frac{2\ell^{N/2}}{N}-(\ell-1)\Bigl(\frac{2\ell^{N/2-1}}{N}-\ell^{N/4}\Bigr)\\
                          &\leqslant \frac{2\ell^{N/2-1}}{N}+\ell^{N/4+1}\,.
  \end{align*}

  For the remaining sum, we need to be more careful (as the value of the polynomials considered at two different elements of $\F_\ell$ are involved). To begin with, the sum can be expressed as follows:
  $$
  \sum_h{\chi_\ell(h(2)h(-2))}=\frac{2}{N}\sum_{\alpha, \deg\alpha=N/2}{\chi_\ell({\rm Norm}((-1)^{N/2}(2-\alpha)(2+\alpha)))}\,,
  $$
  where ${\rm Norm}$ denotes the norm map with respect to the extension $\F_{\ell^{N/2}}/\F_\ell$. Now, using the inclusion-exclusion principle, we get
  $$
  \sum_{\alpha, \deg\alpha=N/2}{\chi_\ell({\rm Norm}(4-\alpha^2))}=\sum_{d\mid N/2}{(-1)^{N/2-d}\sum_{\alpha\in\F_{\ell^d}}{\chi_\ell({\rm Norm}(4-\alpha^2))}}\,.
  $$
  
   For each divisor $d$ of $N/2$, if we set $\chi_{\ell,N}=\chi_\ell\circ {\rm Norm}$ (which is a multiplicative character of $\F_{\ell^d}$), we need to evaluate $\sum_{\alpha\in\F_{\ell^d}}{\chi_{\ell,N}(4-\alpha^2)}$. From the Riemann Hypothesis for curves over finite fields, we derive
   $$
   \Bigl|\sum_{\alpha\in\F_{\ell^d}}{\chi_{\ell,N}(4-\alpha^2)}\Bigr|\leqslant \ell^{d/2}\,,
   $$
since the polynomial $X^2-4$ has distinct roots in $\F_\ell$ (see~\cite[Intro and Th. 1]{KaM} for the statement and the proof of a more general result handling the case of higher dimensional varieties). Using the trivial fact that the number of divisors of $N/2$ is less than $N/2$, we get the upper bound:
   $$
    -\sum_{\alpha, \deg\alpha=N/2}{\chi_\ell({\rm Norm}(4-\alpha^2))}\leqslant \frac{N}{2}\ell^{N/4}\,.
   $$
   
    Thus
    $$
    \sum_h{\chi_\ell(h(2)h(-2))}\geqslant -\ell^{N/4}\,.
    $$
    
    Finally, putting the above estimates together, we get
    $$
    \mt{S}\geqslant \frac{2\ell^{N/2}}{N}-\frac{4\ell^{N/2-1}}{N}-2\ell^{N/4+1}-2\ell^{N/4}\geqslant \frac{2\ell^{N/2}}{N}\Bigl(1-\frac{2(1+N)}{\ell}\Bigr)\,.
    $$

  \end{proof}

  In the last lemma of this subsection, we are interested in counting monic polynomials $h$ of degree $N/2$ with certain imposed factorization patterns and imposed value modulo squares at $2$ and $-2$. Indeed it will be convenient, in the sections to come, to use such information in order to prove the existence of certain elements in the Galois group of an integral polynomial $f$ whose reduction $f\,(\mathrm{mod}\,\ell)$ can be written $f=x^{N/2}h(x+x^{-1})$.

  \begin{lemme} \label{om-tilde-3-4}
  Let $N\geqslant 4$ be an even integer and $\ell\geqslant 5$ be prime. With notation as above (e.g. all the congruences are taken modulo the subgroup of nonzero squares of $\F_\ell$), if we denote
  \par (i) $\tilde{\Theta}_{\ell,3}$ for the set of monic polynomials $f\in M_{N,\ell}$ such that the corresponding $h$ is separable, has an irreducible factor of prime degree $>N/4$ and such that $h(2)\equiv \eps_\ell^{(1)}$, $h(-2)\equiv \eps_\ell^{(2)}$, then we have
  $$
  |\tilde{\Theta}_{\ell,3}|\geqslant   \frac{\ell^{N/2}}{N}\Bigl(\frac{7}{2}-\frac{1}{\ell}(7+\frac{15N}{2})\Bigr)\,,
  $$
  
  (ii) $\tilde{\Theta}_{\ell,4}$ for the set of monic polynomials $f\in M_{N,\ell}$ such that the corresponding $h$ is separable, has a unique irreducible quadratic factor, no other irreducible factor of even degree and such that $h(2)\equiv \eps_\ell^{(1)}$, $h(-2)\equiv \eps_\ell^{(2)}$, then we have
  $$
  |\tilde{\Theta}_{\ell,4}|\gg \frac{\ell^{N/2}}{N}  \,,
  $$
  with an absolute implied constant.
  \end{lemme}

\begin{proof}
\emph{(i)} Let $\ell'$ be a prime such that $N/4<\ell'\leqslant N/2$. The cardinality we are computing is greater than that of the set of monic polynomials $h$ of degree $N/2$ which factor as the product of a monic irreducible polynomial of degree $\ell'$ with imposed values modulo squares at $2$ and $-2$ with any monic irreducible polynomial of degree $N/2-\ell'$ (note that $N/2-\ell'<N/4$ so that no double root may occur in this way). So, applying  Lemma~\ref{adaptcarlitz} (more precisely, using the arguments of the proof) and using once more the lower bound of~\cite[Lemma 3.1]{Ch}, we get
$$
|\tilde{\Theta}_{\ell,3}|\geqslant \frac{\ell^{\ell'}}{4\ell'}\Bigl(1-\frac{2(1+2\ell')}{\ell}\Bigr)\Bigl(\frac{\ell^{N/2-\ell'}}{N/2-\ell'}-\ell^{N/4-\ell'/2}\Bigr)\,.
$$

 As $N/4<\ell'\leqslant N/2$, we have, on the one hand,
 $$
 \frac{\ell^{\ell'}}{4\ell'}\Bigl(1-\frac{2(1+2\ell')}{\ell}\Bigr)\frac{\ell^{N/2-\ell'}}{N/2-\ell'}\geqslant \frac{4\ell^{N/2}}{N}\Bigl(1-\frac{2(1+N)}{\ell}\Bigr)\,,
 $$
 and on the other hand
 $$
 \frac{\ell^{\ell'}}{4\ell'}\Bigl(1-\frac{2(1+2\ell')}{\ell}\Bigr)\ell^{N/4-\ell'/2}\leqslant \frac{\ell^{N/2}}{2N}\Bigl(1-\frac{2+N}{\ell}\Bigr)\,.
 $$
  
   Gathering those two inequalities we get the estimate we wanted to establish.
\par
\medskip
\emph{(ii)} We consider separately the case where $N/2$ is odd and the case where $N/2$ is even. In the former case, the cardinality of $\tilde{\Theta}_{\ell,4}$ is greater than that of the set of polynomials factoring as the product of an irreducible quadratic polynomial having imposed values modulo squares at $2$ and $-2$ with any monic irreducible polynomial of degree $N/2-2$. Thus
$$
|\tilde{\Theta}_{\ell,4}|\geqslant \frac{\ell^2}{8}\Bigl(1-\frac{10}{\ell}\Bigr)\Bigl(\frac{\ell^{N/2-2}}{N/2-2}-\ell^{N/4-1}\Bigr)\,.
$$

 Now, if $N/2$ is even, the set we consider contains all the polynomials which, once divided by their quadratic factor (still with imposed values modulo squares at $\pm 2$) are products of a polynomial of degree $1$ (different from $X\pm2$) with any irreducible polynomial of degree $N/2-3$ (note that if $N=4$, such a polynomial does not exist and, if $N=8$, the other factor of degree $1$ as well as the polynomials $X\pm 2$ must be removed from the set from which that polynomial of degree $N/2-3$ is picked). Thus, using the same inequalities as above,
 $$
 \begin{cases}
 \text{if } N=4, &|\tilde{\Theta}_{\ell,4}|\geqslant \frac{\ell^2}{8}\Bigl(1-\frac{10}{\ell}\Bigr)\,, \\
 \text{if } N=8, &|\tilde{\Theta}_{\ell,4}|\geqslant \frac{\ell^2}{8}\Bigl(1-\frac{10}{\ell}\Bigr)(\ell-2)(\ell-3) \,,\\
 \text{if } N\geqslant 12, &|\tilde{\Theta}_{\ell,4}|\geqslant \frac{\ell^2}{8}\Bigl(1-\frac{10}{\ell}\Bigr)(\ell-2)\Bigl(\frac{\ell^{N/2-3}}{N/2-3}-\ell^{(N-6)/4}\Bigr)\,,
 \end{cases}
 $$
so that we get in particular the estimate stated.
\end{proof}

  \subsection{Number of matrices with prescribed characteristic polynomial}
   In our sieving context, we need a result which would be the analogue, in the orthogonal case, of~\cite[Th. 3.5]{Ch}. The point is that we need to know, for a fixed $f\in M_{N,\ell}$, how many matrices in $O(N,\F_\ell)$ have a reversed characteristic polynomial equal to $f$. Towards such a computation, our first task is to show that \emph{there exists at least one} matrix $g\in O(N,\F_\ell)$ such that $P_g=f$. This is, at least, a crucial step of the proof if we try to follow Chadarov's method. Unfortunately his proof relies heavily on the fact that the symplectic group ${\bf Sp}(2g)$ (as an algebraic group over $\F_\ell$) is simply connected and thus (by a Theorem of Steinberg), that the centralizer under ${\bf Sp}(2g)$ of any semisimple element in $Sp(2g,\F_\ell)$ is a connected algebraic group to which Lang's rationality Theorem may be applied. As we already mentioned neither ${\bf O}(N)$ nor ${\bf SO}(N)$ is a simply connected algebraic group. As a matter of fact, we can easily construct examples of polynomials in $M_{N,\ell}$ which \emph{are not} the reversed characteristic polynomials of any matrix in $O(N,\F_\ell)$. Take, e.g. the polynomial $f(T)=T^2+T+1\in M_{2,\ell}$ and suppose that the quadratic structure on the ambiant space $\F_\ell\times\F_\ell$ is given by the standard scalar product: $\Phi((x_1,y_1),(x_2,y_2))=x_1x_2+y_1y_2$. A straightforward computation shows that any matrix in $O(2,\F_\ell)$ with $f$ as reversed characteristic polynomial must have its non diagonal coefficients equal to half a square root of $3$, and this is obviously not always possible for matrices with coefficients in $\F_\ell$ (problems already occur for $\ell=5$...).
  \par
  While the direct adaptation of Chavadarov's method to the orthogonal case seems to be hopeless, we can however use a result of Baeza (see~\cite{Ba}) that gives a very useful criterion, in the case where the dimension is even, to decide whether an $f\in  M_{N,\ell}$ is or is not the reversed characteristic polynomial of a $g\in O(N,\F_\ell)$. From Baeza's result, we derive the following proposition:

  \begin{proposition} \label{exist_ss}
 Let $N$ be even and $f\in M_{N,\ell}$ such that
 \begin{enumerate}
 \item $f$ is monic,
\item $f$ is separable,
\item $\disc(f)=\disc(Q)$.
 \end{enumerate}
 Then there exists a semisimple element $A\in SO(N, \F_\ell)$ such that 
  $$
  \det(1-TA)=f\,.
  $$
\end{proposition}

  The equality of discriminants (condition ($3$)) is seen as an equality of residue classes in $\F_\ell^\times/(\F_\ell^\times)^2$. This condition is crucial and we easily see that, in our counterexample, $\disc(P)=-1$ while $\disc(Q)=1$ in the case where $\ell=5$.

  \begin{proof}
  From~\cite[Th. (3.7)]{Ba}, we know that the quadratic forms $Q'$ on $V$ such that there exists a $\sigma\in SO(V,Q')$ satisfying
 $$
 \det(1-T\sigma)=\det(\sigma-T)=f(T)\,,
 $$
are exactly those that can be written $s_*(K,n)$, in the notation of loc. cit. We briefly recall how these quadratic spaces are constructed: the separable $\F_\ell$-algebra $K=\F_\ell[T]/(f(T))=\F_\ell[x]$ is equipped with the involution
 $$
 x\mapsto x^{-1}\,.
 $$
 
 If we consider the subalgebra $L=\F_\ell(x+x^{-1})$, then we have a norm map
  $$
  n: K\ra L\,,
  $$
  that defines a non degenerate quadratic form $(K,n)$ with coefficients in $L$. For any trace map $s: L\ra \F_\ell$, we can then consider the composition $s\circ n: K\ra \F_\ell$; it defines a non degenerate quadratic form on $K$. We denote $s_*(K,n)$ the quadratic space obtained (see~\cite[discussion following Prop. 3.6]{Ba} and the references therein for details, namely concerning trace maps).
  
  \par
  \medskip
  For such a fixed quadratic space $s_*(K,n)$, let us consider a $\sigma\in SO(s_*(K,n))$ such that 
  $$
  \det(1-T\sigma)=\det(\sigma-T)=f(T)\,.
  $$
  
   From~\cite[Th. 1.2]{Ba}, we then have $\disc(f)=\disc(s_*(K,n))$ thus, by assumption, $\disc(s_*(K,n))=\disc(Q)$. But we know that quadratic forms over $\F_\ell$ are classified by their discriminant: so $s_*(K,n)\simeq (V,Q)$. Finally, to the element $\sigma$ corresponds a matrix $A\in SO(V,Q)\simeq SO(N,\F_\ell)$ such that 
   $$
   \det(1-TA)=f(T)\,.
   $$
   
   The semisimplicity of $A$ is obvious from the separability assumption on $f$.
  \end{proof}

  In that proof, we see how the distinction between the two models of orthogonal groups (in the case $N$ is even) naturally appears. In particular we notice that, with the notation of Baeza, the quadratic space $s_*(K,n)$ is precisely the one chosen by Katz in~\cite[Section 6]{Ka} to describe a model for the nonsplit orthogonal group.

  We are now ready to prove the main result of this section. Provided the assumptions of Proposition~\ref{exist_ss} are verified, the statement is the analogue, in the orthogonal case, of~\cite[Th. 3.5]{Ch} and it can be interpreted as a property of equidistribution of characteristic polynomials of orthogonal matrices among the polynomials of $M_{N,\ell}$. Apart from the use of Proposition~\ref{exist_ss}, the arguments developped in the proof are quite close to those of loc. cit.

  \begin{theorem} \label{3.5_chav}
  Let $N$ be even and let $f\in M_{N,\ell}$ be such that
\begin{enumerate}
 \item $f$ is monic,
\item $f$ is separable,
\item $\disc(f)=\disc(Q)$.
\end{enumerate}

 Then, 
  $$
  \Bigl|\Bigl\{B\in O(N,\F_{\ell})\mid f(T)=\det(1-TB)\Bigr\}\Bigr|\gg \ell^{N^2/2-N}\,,
  $$
 with an implied constant independent of $N$.
 \end{theorem}

\begin{proof}
 Let $A$ be a semisimple element of $SO(N,\F_\ell)$ with reversed characteristic polynomial equal to $f$ (the existence of such an $A$ is justified by Proposition~\ref{exist_ss}). Let  
 $$
 \Delta=\{B\in O(N,\F_\ell)\mid \det(1-TB)=f\}\,.
 $$
 
  If $B\in \Delta$, its Jordan decomposition can be written
  $$
  B=B_sB_u\,,\,\, B_sB_u=B_uB_s\,,\,\, B_s, B_u \in O(N,\F_\ell)\,,
  $$
  with $B_s$ semisimple and $B_u$ unipotent.
  
  \par
  In particular $\det(1-TB_s)=\det(1-TA)$, therefore the set of matrices $B_s$ satisfying $\det(1-TB_s)=f(T)$ contains the set of all semisimple matrices which are $SO(N,\F_\ell)$-conjugate to $A$. Hence the lower bound
  $$
  |\Delta|\geqslant \{(B_s,B_u)\mid B_s\,\, SO(N,\F_\ell)\text{-conjugate to}\, A\,, B_u\in (C_{SO(N)}(B_s))_u(\F_\ell)\}\,,
  $$
  where $C_{SO(N)}(B_s)$ (resp. $(C_{SO(N)}(B_s))_u$) denotes the centralizer of $B_s$, seen as an algebraic group, under the action of ${\bf SO}(N)$ (resp. the unipotent part of that centralizer).
  
  \par
  \medskip
  As already mentioned, we can not guarantee that the algebraic group $C_{SO(N)}(A)$ is connected and we will denote $C_{SO(N)}(A)^0$ its (connected) identity component. Then we argue as in~\cite[proof of Th. 3.5]{Ch} to obtain
  $$
  |\{\text{Unipotent elements of}\,\, (C_{SO(N)}(A))^0(\F_\ell)\}|=\ell^{\dim (C_{SO(N)}(A))^0-\rank(C_{SO(N)}(A))^0}\,.
  $$
  
   From~\cite[II.12.2, prop.]{Bo}, $C_{SO(N)}(A)^0$ is a maximal torus in ${\bf SO}(N)$ (indeed ${\bf SO}(N)$ is a reductive group and from the separability assumption on $f$, we know $A$ is regular semisimple), thus 
   $$
   \rank(C_{SO(N)}(A))^0=\rank SO(N)=\frac{N}{2}\,.
   $$
   
   This finally yields the lower bound:
 $$
 |\Delta|\geqslant \ell^{\dim (C_{SO(N)}(A))^0-N/2}\frac{|SO(N,\F_\ell)|}{|C_{SO(N)}(A)(\F_\ell)|}\,.
 $$
 
  Moreover, a Theorem of Steinberg asserts that the group of connected components of an algebraic group is always a subgroup of its fundamental group (see~\cite[II Cor. 4.4]{SpSt}). We deduce that $C_{SO(N)}(A)$ has at most $2$ connected components. Adapting the result of Nori (see~\cite{No}) used by Chavdarov~(\cite[page 160]{Ch}), we obtain
  $$
  |C_{SO(N)}(A)(\F_\ell)|\leqslant 2(\ell+1)^{\dim (C_{SO(N)}(A))^0}\,.
  $$  
  
  Thanks to the formula (see, e.g.,~\cite[page 147]{Ar.E}) on the cardinality of the special orthogonal group over $\F_\ell$ in the even dimensional case, we deduce
  $$
  (\ell-1)^{N(N-1)/2}\leqslant |SO(N,\F_\ell)|\leqslant \ell^{N(N-1)/2}\,.
  $$
  
  Combining those last inequalities, we get 
  $$
  |\Delta|\gg \ell^{-N/2}(\ell-1)^{N(N-1)/2}\,,
  $$
  thus $|\Delta|\gg \ell^{N^2/2-N}$, where, in these last two inequalities, the implied constant does not depend on $N$.
\end{proof}

 The purpose of the last result we give in this section is to relate the cardinalities of a given conjugacy invariant subset of $O(N,\F_\ell)$ and of the set of corresponding reversed characteristic polynomials in $M_{N,\ell}$. To that extent, its main interest lies in how we can apply it to our sieving problem. The arguments being very close to those used in the proofs of Theorem~\ref{3.5_chav} and Proposition~\ref{exist_ss}, it seems fair to include it in this ``independent section''.

 \begin{lemme} \label{lem7.2ko1}
 Let $N\geqslant 2$ be an integer and $\ell$ be an odd prime number. Consider a subset $\tilde{\Theta}_\ell$ with cardinality $\tilde{\theta}_\ell$ of $M_{N_{red},\ell}$ such that the elements $f\in \tilde{\Theta}_\ell$ satisfy 
\begin{enumerate}
 \item $f$ is monic,
 \item $f$ is separable,
 \item $\disc(f)=\disc(Q)$, 
\item either $f(-1)$ is a nonzero square of $\F_\ell$ for every $f\in \tilde{\Theta}_\ell$, or $f(-1)$ is a non square for every $f\in\tilde{\Theta}_\ell$.
\end{enumerate}

 Moreover let $\eps_\ell^{(1)}, \eps_\ell^{(2)}$ be two elements of $\F_\ell$ each equal to $\pm 1$ and such that $\eps_\ell^{(2)}$ ``is'' the residue class in $\F_\ell^\times/{\F_\ell^\times}^2$ defined by condition $(3)$ above and Lemma~\ref{lem_edwards}. Let
$$
\Theta_\ell=\{g\in O(N,\F_\ell)\mid (\det, \nsp)(g)=(\eps_\ell^{(1)},\eps_\ell^{(2)}),\,\det(1-Tg)_{red}\in \tilde{\Theta}_\ell\}\,.
$$

 Then, if $\theta_\ell=|\Theta_\ell|$, we have
$$
\theta_\ell|\Omega(N,\F_\ell)|^{-1}\geqslant \tilde{\theta}_\ell \ell^{-N_{red}/2}\Bigl(1-\frac{1}{\ell+1}\Bigr)^{N_{red}(N_{red}-1)/2}\,.
$$
\end{lemme}

 At first, the above statement can look ambiguous as the integer $N_{red}$ is not entirely defined by $N$ but also depends on the matrix we consider. The point is, that once the determinant is fixed (which is the case in the Lemma since we restrict ourselves to matrices with determinant $\eps_\ell^{(1)}$), there is only one integer $N_{red}$ that can correspond to $N$, so that, a posteriori, the assertion of the Lemma makes sense.

 \begin{proof}
 Let us first consider the auxiliary set 
\begin{equation} \label{hs}
\{h\in O(N_{red},\F_\ell)\mid (\det, \nsp)(h)=(1,\tilde{\eps}_\ell^{(2)}), \det(1-Th)=\det(1-Th)_{red}\in\tilde{\Theta}_\ell\}\,,
\end{equation}
where $\tilde{\eps}_\ell^{(2)}$ is a fixed element of $\{1,\eps_0\}$ (which still denotes a set of representatives for $\F_\ell/{\F_\ell^\times}^2$). We can trivially inject the set~(\ref{hs}) in $\Theta_\ell$ via the map $h\mapsto h\oplus u$ where $u$ is any representative of a fixed class of $O(N-N_{red},\F_\ell)$ modulo $\Omega(N-N_{red},\F_\ell)$ (that class corresponding to the couple $(x,y)\in\{\pm 1\}$ such that $(1,\tilde{\eps}_\ell^{(2)})\times (x,y)=(\eps_\ell^{(1)},\eps_\ell^{(2)})$) and the quadratic structure on the corresponding $(N-N_{red})$-dimensional being chosen with discriminant $1$. Imbedding~(\ref{hs}) in $\Theta_\ell$ that way, we end up with a $N_{red}$-dimensional quadratic space having the same discriminant as the ambiant $N$-dimensional quadratic space $(V,Q)$. Moreover, from a fixed $h$ in the set~(\ref{hs}) we construct $(\Omega(N,\F_\ell):\Omega(N_{red},\F_\ell))$ distinct elements of $\Theta_\ell$. So, following the same idea as in~\cite[Lemma 7.2]{KoCrelle}, we can now compute a lower bound for $\theta_\ell$ involving $\tilde{\theta}_\ell$. First, we have:
   $$
   \theta_\ell\geqslant \frac{|\Omega(N,\F_\ell)|}{|\Omega(N_{red},\F_\ell)|} \sum_{f\in{\tilde{\Theta}_\ell}}{|\{g\in O(N_{red},\F_\ell)\mid (\det, \nsp)(g)=(1,\eps_\ell^{(2)}), \det(1-Tg)=f\}|}\,,
  $$

 We know, thanks to the assumption $(3)$, that Proposition~\ref{exist_ss} can be applied and so, that each summand of the right hand side of the above inequality is nonzero. More precisely each of these quantities is equal to the cardinality of the set $\Delta$ (depending on the polynomial $f$) of the proof of Theorem~\ref{3.5_chav}. Following that proof and using the above inequality, we get
  $$
\theta_\ell\geqslant\ell^{-N_{red}/2} \frac{|\Omega(N,\F_\ell)|}{|\Omega(N_{red},\F_\ell)|}\sum_{f\in\tilde{\Theta}_\ell}{\frac{\ell^{d_f}|SO(N_{red},\F_\ell)|}{|C_{SO(N_{red})}(A_f)(\F_\ell)|}}\,,
$$ 
where, for each $f\in \tilde{\Theta}_\ell$, the matrix $A_f$ is the semisimple element whose existence is guaranted by Proposition~\ref{exist_ss} , $C_{SO(N_{red})}(A_f)$ denotes the centralizer of $A_f$ under the action of ${\bf SO}(N_{red})$ and $d_f$ is the dimension of the identity component of that algebraic group. The proof of Theorem~\ref{3.5_chav} yields 
$$
|C_{SO(N_{red})}(A_f)|\leqslant 2(\ell+1)^{d_f}\,.
$$ 

 Now the derived group $\Omega(N_{red},\F_\ell)$ has index $2$ in $SO(N_{red},\F_\ell)$ so we have 
$$
\frac{\theta_\ell}{|\Omega(N,\F_\ell)|}\geqslant 2\ell^{-N_{red}/2} \sum_{f\in\tilde{\Theta}_\ell}{\frac{\ell^{d_f}}{2(\ell+1)^{d_f}}}\,.
$$

Thus
\begin{align*}
\theta_\ell|\Omega(N,\F_\ell)|^{-1}&\geqslant \ell^{-N_{red}/2}\sum_{f\in\tilde{\Theta}_\ell}{\Bigl(1-\frac{1}{\ell+1}\Bigr)^{d_f}}\\
                                  &\geqslant \ell^{-N_{red}/2} \Bigl(1-\frac{1}{\ell+1}\Bigr)^{N_{red}(N_{red}-1)/2}\tilde{\theta}_\ell\,, 
\end{align*}
since we have $d_f\leqslant \dim SO(N_{red})=N_{red}(N_{red}-1)/2$.
 
\end{proof}

\begin{remark}
 One can wonder why, in all the computations performed in this section, we always gave uniform bounds (with respect to the parameter $N$) for the quantities studied rather than asymtotic estimates which, most likely, would have been easier to establish and would suffice for the argument needed in the proof of Theorems~\ref{exp_decrease} and~\ref{exp_decrease_gen}. This is because we have in mind another possible application (that we postpone to a future paper) involving $L$-functions of families of elliptic curves over function fields. In that work, we will show how the large sieve arguments used here can yield a lower bound for the proportion of elliptic curves with irreducible (up to trivial factors) $L$-function (seen as a $\Q$-polynomial), when the curve varies in a suitable algebraic family, which is uniform with respect to the common conductor of the family (provided the related estimates for the local densities involved are uniform as well).
\end{remark}


 \section{Statement and proof of the main result} \label{applications}
 In this last section, we state the main result of this paper which generalizes Theorem~\ref{exp_decrease} in two different ways. To that purpose we show that for the two different kinds of group considered, Proposition~\ref{borneconstconj} and Proposition~\ref{exp_sum_maj} hold. Then, to derive our results from the large sieve inequality~(\ref{inegfond}), we need to find a suitable lower bound for the constant $H$. That issue, in the case where the groups involved are orthogonal groups, can be handled thanks to the results of Section~\ref{localdensities}. In the other case we consider ($G=GL(n,A)$, $G^g=SL(n,A)$), the question of finding a lower bound for $H$ turns out to be easier, as we do not have as many constraints on the matrices considered as in the former case. 
 \par
  Let us first explain the additional terminology needed to state our result in full generality. Indeed, we need not (as Theorem~\ref{exp_decrease} would suggest) restrict ourselves to the study of the irreducibility of characteristic polynomials of ``random matrices'', but the sieve setting we are working with enables us to study the Galois group of those $\Q$-polynomials (see~\cite[Chap. 7]{KoLS} for the analogous study for the groups $SL(n,\Z)$ and $Sp(2g,\Z)$). The Galois group of the (splitting field over $\Q$ of) the characteristic polynomial of a matrix of $GL(n,\Z[1/\mt{P}])$ can a priori be as big as the whole symmetric group $\mathfrak{S}_n$, but in the case of orthogonal matrices,~(\ref{eqfonc}) imposes conditions on the roots: if we denote $N=n+m$ and $N_{red}=2\lfloor(n+m)/2\rfloor$, the biggest subgroup of $\mathfrak{S}_{N_{red}}$ that can be realized as the Galois group of the characteristic polynomial of a matrix $g\in SO(n,m)(\Z)$ is the group denoted $W_{N_{red}}$ which can be seen as the subgroup of $\mathfrak{S}_{N_{red}}$ permuting $N_{red}/2$ pairs of elements of $\{1,2,\ldots,N_{red}\}$. Of course if the Galois group of the characteristic polynomial of an element $g\in G$ is maximal (i.e. is equal to $\mathfrak{S}_n$ if $G=GL(n,\Z[1/\mt{P}])$ or to $W_{N_{red}}$ if $G=SO(n,m)(\Z)$) then the polynomial is irreducible.
  
  \par \medskip
  We can also state a generalized version of the second part of Theorem~\ref{exp_decrease}. To do so, it is convenient to use (some of the basics of) the language of logic (as done in~\cite{KoIs}). Recall (see Section $2$ of loc. cit.) that a \emph{term} in the language of rings is simply a polynomial $f\in\Z[x_1,\ldots,x_n]$ and that an \emph{atomic formula} $\varphi$ is a formula of the form $f=g$ where $f$ and $g$ are polynomials (non necessarily in the same variables). Now if $\varphi$ is an atomic formula, $A$ is a ring and if we assign the family of elements $a=(a_i)$ to the set of variables involved in the definition of $\varphi$, we say that \emph{$\varphi(a)$ is satisfied in $A$} and we denote
  $$
  A\models\varphi(a)\,,
  $$
 if the equality which ``is'' $\varphi$ is satisfied in $A$ when the variables are given the values $a_i$. From atomic formul\ae\, we can build the so-called \emph{first order formul\ae} by induction, using the symbols $\neg$, $\vee$, $\wedge$ and the quantifiers $\exists$, $\forall$ (we refer the reader to loc. cit. for examples of quite complicated formul\ae\, that can be obtained in this way). Next, if $\varphi(x)$ is a first order formula with respect to the variables $x=(x_1,\ldots,x_n)$ and if $A$ is a ring, then we denote
 $$
 \varphi(A)=\{x\in A^n\mid A\models \varphi(x)\}\,.
 $$
 
  With such a terminology, we can state the main result of this paper:

 \begin{theorem} \label{exp_decrease_gen}
 With the above notation (as well as those used in the introduction) and assuming that the first condition of Proposition~\ref{keyprop} holds and that $n+m\geqslant 6$ (resp. $n\geqslant 2$) in the case $G=SO(n,m)(\Z)$ (resp. $G=GL(n,\Z[1/\Pp]$), we have
 \par (i) there exists a $\beta_3>0$ such that for all $k\geqslant 1$, we have
  $$
  {\bf P}(\det(T-X_k)_{red}\in A[T] \text{ is reducible or does not have maximal Galois group})\ll \exp(-\beta_3 k)\,,
  $$
  with $\beta_3$ depending only on the underlying algebraic group $\G/\Q$, on the generating set $S$ and on the sequence $(p_s)_s$ (i.e. on the distribution of the $\xi_k$) and where ``maximality of the Galois group'' means ``with Galois group over $\Q$ equal to $\mathfrak{S}_n$ (resp. $W_{N_{red}}$)'' if $G=GL(n,\Z[1/\mt{P}])$ (resp. $G=SO(n,m)(\Z)$ and $N_{red}=2\lfloor (n+m)/2\rfloor$). Moreover the implied constant depends only on $\G$ and the density of $\Pp$ in the set of all rational prime numbers (in the case where $G=GL(n,\Z[1/\Pp]))$.
  \par
 (ii) Let $\varphi$ be a first order formula in the language of rings with respect to the variables $x=(x_{i,j})_{1\leqslant i,j\leqslant N}$ (where $N=n$ if $G=GL(n,\Z[1/\Pp])$ and $N=n+m$ if $G=SO(n,m)(\Z)$). Set
 $$
 \Lambda_\delta=\big\{\ell \text{ prime}\mid |\big(\neg\varphi(\F_\ell)\big)\cap Y_\ell |\cdot|G_\ell^g|^{-1}\geqslant \delta\big\}\,,
 $$
 and assume
 \begin{equation} \label{cond}
 \text{there exists }\delta>0\text{ such that } \Lambda_\delta\text{ has strictly positive Dirichlet density }\,,
 \end{equation}
 then there exists a $\beta_4>0$ such that for all $k\geqslant 1$, we have
  $$
  {\bf P}(A\models \varphi(X_k))\ll \exp(-\beta_4 k)\,,
  $$
  with the same dependency for $\beta_4$ as for $\beta_3$ and the same dependency for the implied constant as in the previous case.
 \end{theorem} 

\begin{remarks}
$(i)$ If $G$ is the subgroup of integral points of a special orthogonal group, the fact that we emphasize (e.g. in Theorem~\ref{exp_decrease}) the case where the quadratic structure is hyperbolic (i.e. the signature of the corresponding form is $(n,n)$) comes from the first condition of Proposition~\ref{keyprop}. Indeed, in the hyperbolic case, that condition is always fulfuilled (as we will see in Lemma~\ref{lem7.3}). Another (somewhat artificial) way to ensure we can find in the general case a relation of odd length among the elements of a generating set $S$ could be to add the identity element to $S$. Doing  so we would end up with a ``lazy'' random walk for which we would have at each step a probability $p_{{\rm Id}}>0$ to stay at the same point. 
\par
 The same problem occurs in the case where $G=SL(2,\Z[1/\Pp])$ (and this explains why this case is omitted in the statement of Theorem~\ref{exp_decrease}); indeed, while the periodicity condition is always fulfilled for any generating system $S$ of $SL(n,\Z[1/\Pp])$ if $n\geqslant 3$ (as will be proved in Lemma~\ref{lem7.3bis}$(ii)$), there are examples of such sets $S$ for which that property does not hold in the case $n=2$ (see~\cite[Section 7.4]{KoLS} for further details).
\par
$(ii)$ The group denoted $W_{N_{red}}$ that comes into play in the case of orthogonal groups has a more ``functorial'' description as the one given above. Indeed it is the so-called \emph{Weyl group} of the algebraic group ${\bf SO}(N)$. The fact that the Galois groups we investigate can be embedded in the Weyl group of the underlying algebraic group seems to be a quite general fact (it is proven by Kowalski for ${\bf SL}(n)$ and ${\bf Sp}(2g)$ in~\cite[Chap. 7]{KoLS} and the case of (the split form of) the exceptional group $E8$ is treated in~\cite{JKZ}). 
\end{remarks}

 Before getting into the details of the proof, let us give a few more remarks on part $(ii)$ of the statement. First, for a fixed first order formula $\varphi$, the set
 $$
 \big(\neg\varphi(\F_\ell)\big)\cap Y_\ell
 $$
(recall $Y_\ell=\rho_\ell(\alpha)G_\ell^g$) in the above statement is in fact the sieving set $\Theta_\ell$ with index $\ell$, in the notation of~(\ref{inegfond}) and of the appendix. Next, to deduce the second part of Theorem~\ref{exp_decrease} from $(ii)$ of Theorem~\ref{exp_decrease_gen}, we choose for $\varphi$ the formula:
 $$
 \varphi(x):\bigvee_{1\leqslant i,j\leqslant N}{\exists y, y^2=x_{i,j}}\,.
 $$ 
 
 Thus, assuming Theorem~\ref{exp_decrease_gen}, proving the second part of Theorem~\ref{exp_decrease} is equivalent (once shown that the hypotheses of Proposition~\ref{exp_sum_maj} are satisfied for the groups we study) to the fact that~(\ref{cond}) holds for that choice of $\varphi$. 
 
 \par
 Finally, let us give examples of situations (i.e. choices of $\varphi$) for which~(\ref{cond}) holds/does not hold: consider for instance the case where $\varphi(\F_\ell)$ is the set of $\F_\ell$-rational points of a subvariey $V/\F_\ell$ with codimension $\geqslant 1$ in ${\bf G}$. Then~(\ref{cond}) clearly holds since $|V(\F_\ell)|$ is trivially bounded by $C\ell^{\dim V}$ where $C$ is an absolute constant. In the opposite direction, if we investigate the probability with which the trace of a matrix in $\alpha G^g$ is sum of two squares, we quickly see that our method does not yield any quantitative information at all: indeed, if $\ell$ is an odd prime number, any element in $\F_\ell$ is the sum of two squares (a quite classical application of the pigeonhole principle) so that~(\ref{cond}) is false for the choice
 $$
 \varphi((x_{i,j})): \exists a, \exists b, \sum_{i=1}^N{x_{i,i}}=a^2+b^2\,.
 $$

 \par\medskip
 The remaining sections are devoted to the proof of Theorem~\ref{exp_decrease} and its generalization Theorem~\ref{exp_decrease_gen}. In order to handle the case of orthogonal groups, we first review some useful facts concerning certain quadratic modules over $\Z$.

\subsection{Quadratic modules over $\Z$}  \label{quad-mod-Z}
  Let $n,m$ be integers such that $n+m\geqslant 4$ and let $(M,Q)$ be a quadratic module over $\Z$ with signature $(n,m)$. The notion of spinor group we reviewed at the beginning of Section~\ref{localdensities} can be extended to the case of quadratic modules (see~\cite[Section 7.2A]{HOM}): in that more general case, we still have a morphism
  $$
  \spin(M)\ra O(M)\,,  
  $$
 the image of which is denoted $\Omega(M)$ and the kernel of which is  $\pm 1$ (see~\cite[Th. 7.2.21]{HOM}). In other words, these groups fit the exact sequence
 \begin{equation} \label{suite_ex_modules}
 1\ra\{\pm 1\}\ra \spin(M)\ra \Omega(M)\ra 1\,,
 \end{equation}
 where $\Omega(M)$ can once more be seen as the simultaneous kernel in $O(M)$ of the determinant and the spinor norm (as defined in~\cite[p 419]{HOM}). In the arithemtic context we are interested in we will respectively denote by $O(n,m)(\Z)$, $SO(n,m)(\Z)$	and $\Omega(n,m)(\Z)$ for $O(M)$, $SO(M)$ and $\Omega(M)$. Moreover, it will be convenient sometimes to see the first two of these groups as the groups of integral points of the algebraic groups ${\bf O}(n,m)/\Q$ and ${\bf SO}(n,m)/\Q$ respectively.  The properties we need $\Omega(n,m)(\Z)$ to verify, in order to apply Propositions~\ref{borneconstconj} and~\ref{exp_sum_maj} are contained in the following lemma

 \begin{lemme} \label{lem7.3}
 $(i)$ For integers $n,m$ such that $n+m\geqslant 4$, we have
 \begin{enumerate}
 \item If $d\geqslant 1$ is a squarefree integer whose only prime factors are outside a fixed finite set $\mt{S}$, then the reduction modulo $d$: $\Omega(n,m)(\Z)\ra \Omega(n,m)(\Z/d\Z)$ is onto.
 \item Property $(\tau)$ holds for $\Omega(n,m)(\Z)$ with respect to the family of its congruence subgroups $\bigl(\ker(\rho_d:\Omega(n,m)(\Z)\ra \Omega(n,m)(\Z/d\Z))\bigr)_{d\geqslant 1}$.
 \end{enumerate}
 $(ii)$ If in addition the quadratic module considered is hyperbolic (in the sense of the introduction, in particular $n=m$), and if $n\geqslant 3$, we have
 \begin{enumerate}
 \item $\Omega(n,n)(\Z)$ is finitely generated.
 \item For every symmetric generating system $S$ of $\Omega(n,n)(\Z)$ there exits a relation of odd length $c$ inside $S$:
 $$
 s_1\cdots s_c=1\,,\, s_i\in S\,.
 $$
 \end{enumerate}
 \end{lemme}

 \begin{proof}
 For $(ii)$ we use the useful elements of the automorphism group of $O(n,n)(\Z)$ called \emph{Eichler transformations} (see~\cite[5.2.9]{HOM}). As the ambiant quadratic module $(M,Q)$ we consider here is hyperbolic, we know from Theorem $9.2.14$ of loc. cit. that these transformations span the subgroup $\Omega(n,n)(\Z)$. This proves $(ii) (1)$. Moreover the same result asserts that the \emph{unitary elementary transformations} $E_{i,j}(1)$, where $i\not =j$ run over a finite set of indices, span the group $\Omega(n,n)(\Z)$ and satisfy the commutator relation
 $$
 [E_{i,j}(1): E_{k,l}(1)]=E_{i,l}(1)\,,
 $$
if $i,j,k,l$ are distinct and run over the same (suitably chosen) set of indices (see~\cite[Th. 9.2.14]{HOM}). Hence $\Omega(n,n)(\Z)$ equals its own derived group. Now to prove $(ii) (2)$ let us assume, by contradiction, that there is no relation of odd length among a fixed symmetric generating system $S$ of $\Omega(n,n)(\Z)$ and let $F(S)$ denote thre free group generated by $S$. The morphism
 \begin{align*}
 F(S) &\ra \{\pm 1\}\\
 s &\mapsto -1\,,
 \end{align*}
 induces a surjective morphism $\Omega(n,n)(\Z)\ra \{\pm 1\}$. Thus a quotient of $\Omega(n,n)(\Z)$ is isomorphic to $\{\pm 1\}$; this contradicts the fact that $\Omega(n,n)(\Z)$ is its own derived group.
 
 \par
 \medskip
  For $(i) (1)$ we first use strong approximation to justify the surjectivity of the reduction
  $$
  \pi_p: \spin(n,m)(\Z)\ra \spin(n,m)(\F_p)\,,
  $$
  where $p$ runs over the set of all prime numbers which do not lie in a fixed finite set $\mt{S}$. Indeed, using the fact that the group $\spin(n,m)(\R)$ is not compact (the algebraic group ${\bf Spin}(n,n)$ is said to be \emph{of non compact type}), we can apply Borel's Density Theorem (see~\cite[Th. 4.10]{PR}). Now $\spin(n,m)(\Z)$ is a Zariski dense subgroup of the simply connected algebraic group ${\bf Spin}(n,m)/\Q$, so strong approximation can be applied; more precisely, thanks to~\cite[Th. 7.15]{PR}, we deduce the surjectivity of $\pi_p$ provided $p$ remains outside of a finite set $\mt{S}$ of prime numbers. 
  \par
  Then, using~(\ref{suite_ex_modules}), we can consider the diagram with exact rows
  $$
   \begin{CD}
1 @>>> \{\pm 1\} @>>> \spin(n,m)(\Z) @>>> \Omega(n,m)(\Z) @>>> 1\\
@. @VV \pi_p V @VV \pi_p V @VVV\\
1 @>>> \{\pm 1\} @>>> \spin(n,m)(\F_p) @>>> \spin(n,m)(\F_p)/\{\pm 1\} @>>> 1\,.
\end{CD}
  $$
  
  If $p$ is a prime not lying in $\mt{S}$, the two left vertical arrows of the above diagram are onto (as we have assumed to be working only with odd prime numbers). Then, if we define the last vertical arrow in such a way that the diagram commutes, this map must be onto as well. Moreover, it is easily seen that this last arrow also corresponds to the usual reduction modulo $p$
  $$
  \Omega(n,m)(\Z)\ra \Omega(n,m)(\F_p)\,.
  $$
  
  Now, for a squarefree integer $d$ without any prime factor in $\mt{S}$, we invoke Goursat-Ribet's Lemma (as stated in~\cite[Prop. 5.1]{Ch}). Indeed, for $p\not\in\mt{S}$, the group $\Omega(n,m)(\F_p)$ has no non central proper normal subgroup and the group $\Omega(n,m)(\F_p)$ modulo its center is a simple group. So, for such a $d$, we have a surjective morphism
  $$
  \Omega(n,m)(\Z)\ra \Omega(n,m)(\Z/d\Z)\,.
  $$ 
  
  \par
  \medskip
  Finally, for $(3)$, we can apply~\cite[Th. 8, page 23]{HV}, thanks to which we know that, as $n+m\geqslant 4$, the group $SO(n,m)(\R)$ has Kazhdan's Property $(T)$. Combining this with~\cite[Cor. 5 to Th. 4 page 33]{HV} we deduce first that $SO(n,m)(\Z)$ has Property $(T)$ and then that $\Omega(n,m)(\Z)$ also has Property $(T)$. The weaker Property $(\tau)$ fot $\Omega(n,m)(\Z)$ with respect to the family of its congruence subgroups follows immediately.
 \end{proof}

 \begin{remark}
 In the above proof, we use the notation $\spin(n,m)(\F_p)$ or $\Omega(n,m)(\F_p)$ just to keep track of the indefinite quadratic form over $\Q$ giving rise to the matrix groups for which we then take the reduction modulo $p$.
 \end{remark}

 \subsection{Proof of the main Theorem}
  
  Recall that we are working out the two following cases:
 \par $\bullet$ Either, for $n+m\geqslant 6$ and with notation as above, $G=SO(n,m)(\Z)$, $G^g=\Omega(n,m)(\Z)$, and $\Gamma=SO(n,m)(\Z)/\Omega(n,m)(\Z)$, in which case, these groups fit the following commutative diagram with exact rows and surjective \emph{left} vertical map
 $$
  \begin{CD}
1 @>>> \Omega(n,m)(\Z) @>>> SO(n,m)(\Z) @>>> \Gamma @>>> 1\\
@. @VV \pi_p V @VV \pi_p V @VVV\\
1 @>>>\Omega(n,m)(\F_p) @>>> SO(n,m)(\F_p) @>>> \Gamma_p @>>> 1\,,
\end{CD}
 $$
provided $p\not\in \mt{S}$ (see Lemma~\ref{lem7.3}), and where $\Gamma_p$ denotes the abelianization of $SO(n,m)(\F_p)$.

\par
\medskip
$\bullet$ Or, for $n\geqslant 2$ and with notation as in the introduction, we set $G=GL(n,\Z[1/\mt{P}])$, $G^g=SL(n,\Z[1/\mt{P}])$, $\Gamma=\Z[1/\mt{P}]^\times$, in which case these groups fit the following commutative diagram with exact rows 
$$
  \begin{CD}
1 @>>> SL(n,\Z[1/\mt{P}]) @>>> GL(n,\Z[1/\mt{P}]) @>>> \Gamma @>>> 1\\
@. @VV \pi_p V @VV \pi_p V @VVV\\
1 @>>>SL(n,\F_p) @>>> GL(n,\F_p) @>>> \F_p^\times @>>> 1\,,
\end{CD}
 $$
provided $p\not\in \mt{P}$. The question of the surjectivity of the downward arrows $\pi_p$ (which still denote reduction modulo $p$) is easily answered here. Indeed, it is straightforward to check that $GL(n,\Z)\ra GL(n,\F_p)$ is surjective and the fact that its restriction $SL(n,\Z)\ra SL(n,\F_p)$ is also surjective is well known (see, for instance~\cite[Lemma 1.38]{shimura} where the surjectivity is proved in the more general case of the reduction $SL(n,\Z)\ra SL(n,\Z/d\Z)$ modulo any positive integer $d$). As we want to apply Propositions~\ref{keyprop},~\ref{borneconstconj} and~\ref{exp_sum_maj} to this case, we need the following analogue for the points of Lemma~\ref{lem7.3}

\begin{lemme} \label{lem7.3bis}
$(i)$ For $n\geqslant 2$, we have the following properties:
\begin{enumerate}
\item $SL(n,\Z[1/\mt{P}])$ surjects onto $SL(n,\Z/d\Z)$ for each squarefree $d$ without prime factors in $\mt{P}$.
\item $SL(n,\Z[1/\mt{P}])$ has Property $(\tau)$ with respect to the family of its congruence subgroups
$$
\Bigl(\ker (\rho_d: SL(n,\Z[1/\mt{P}])\ra SL(n,\Z/d\Z))\Bigr)_d\,,
$$ 
where $d$ runs over the set of squarefree integers without prime factors in $\mt{P}$.
\end{enumerate}

$(ii)$ If we suppose $n\geqslant 3$ then, for every symmetric generating system $S$ of $SL(n,\Z[1/\mt{P}])$, there exists a relation of odd length $c$:
$$
s_1\cdots s_c=1,\, s_i\in S\,.
$$
\end{lemme}

\begin{proof}
We have just discussed point $(i)(1)$ and $(i)(2)$ is proven in Lemma~\ref{SL-tau}.
\par
 For ($ii$), we note that, as the ring $\Z[1/\mt{P}]$ is euclidian, the group $SL(n,\Z[1/\mt{P}])$ is generated by the (infinite) set of transvection matrices $T_{i,j}(a)$ (the sum of the identity matrix and the matrix with all entries equal to zero except for the one in position $(i,j)$ which equals $a$), where $a\in\Z[1/\mt{P}]^\times$ and $1\leqslant i\not=j\leqslant n$. Such matrices satisfy the commutator relation
 $$
 [T_{i,j}(a_1),T_{j,k}(a_2)]=T_{i,k}(a_1a_2)\,,
 $$
 as soon as $i,j,k$ are pairwise distinct (such a choice is indeed possible since $n\geqslant 3$). From that equality, we deduce that $SL(n,\Z[1/\mt{P}])$ equals its own commutator subgroup. The end of the proof is then exactly the same as for the last point of Lemma~\ref{lem7.3}.
\end{proof}

 We are now ready to prove Theorem~\ref{exp_decrease_gen}. Thanks to Lemmas~\ref{lem7.3} and~\ref{lem7.3bis}, Propositions~\ref{borneconstconj} and~\ref{exp_sum_maj} hold with data corresponding to the two cases described above. To prove the exponential decrease of the probabilities investigated as $k$ grows, we need to give a suitable lower bound for the constant
 $$
 H=\sum_{\ell\in \Lstar}{\frac{\nu_\ell(\Theta_\ell)}{1-\nu_\ell(\Theta_\ell)}}\,,
 $$
 where $\Lstar$ is the set of primes in $\Lambda$ up to a fixed $L\geqslant 1$ and $\Lambda$ is a set of primes with strictly positive Dirichlet density that we will make precise in due course.
 
 
 Both for the conjugacy and the non-conjugacy coset sieve, we have
 $$
 H\geqslant \sum_{\ell\in\Lstar}{\frac{|\Theta_\ell|}{|G_\ell^g|}}\,.
 $$
 
 We shall prove that
 \begin{equation} \label{grandO}
 \frac{|\Theta_\ell|}{|G_\ell^g|}\gg 1\,,
 \end{equation}
 with an implied constant depending only on $n$ and on the underlying algebraic group ${\bf G}$, for our different choices of sieving sets $\Theta_\ell$ and groups $G_\ell^g$. Such an estimate will turn out to be sufficient to prove Theorem~\ref{exp_decrease_gen} (and deduce Theorem~\ref{exp_decrease}).
 \par
 First we note this inequality is quite obvious in the setting of $(ii)$ of Theorem~\ref{exp_decrease_gen} with the choice $\Lambda=\Lambda_\delta$ (the hypotheses are chosen puposely for that estimate to be true). Indeed, in that case, we have $\Theta_\ell=\big(\neg\varphi(\F_\ell)\big)\cap Y_\ell$, so that, by assumption, there exists a $\delta>0$ such that
 $$
 \frac{|\Theta_\ell|}{|G_\ell^g|}\geqslant \delta\,,
 $$
 for all $\ell\in \Lambda_\delta$. Then, looking back at~(\ref{inegfond}) and using Proposition~\ref{exp_sum_maj} as well as the Prime Number Theorem, we can easily get $(ii)$. Indeed, we have
  $$
  {\bf P}\big(\F_\ell\models \varphi(\rho_\ell(X_k))\text{ for all } \ell\leqslant L\big)\ll (1+L^{(17d+4)/4}\exp(-\eta k))L^{-1}\log L\,,
  $$
  with an implied constant depending only on $n$ and the density of $\Lambda_\delta$ as a subset of the rational primes. Setting $L=\exp(\frac{4\eta k}{17d+4})$, we obtain $(ii)$ of Theorem~\ref{exp_decrease_gen}.

 \par
 \medskip
  Now as far as $(i)$ of Theorem~\ref{exp_decrease_gen} is concerned, the set $\Lambda$ we choose is, depending on the type of group considered, either the set of primes which are not in $\mt{S}$ (see Lemma~\ref{lem7.3}) or the complementary of $\mt{P}$ in a subset (with finite complement as we will see later) of all prime numbers (see Lemma~\ref{lem7.3bis}). The conjugacy coset sieve enables us to study the reduced characteristic polynomial of $X_k$ as $k$ grows. For a fixed $\alpha\in G$, we choose:
 \begin{equation} \label{choix_enscriblants}
\Theta_\ell=\{g\in \rho_\ell(\alpha)G_\ell^g\mid \det(T-g)_{red}\in \tilde{\Theta}_\ell\}\,,
 \end{equation}
 where $\tilde{\Theta}_\ell$ is, for each $\ell$, a set of polynomials having imposed factorisation patterns. Each $\Theta_\ell$ will be a conjugacy invariant subset of $Y_\ell$. Moreover, in the case where $G=SO(n,m)(\Z)$ and $G^g=\Omega(n,m)(\Z)$, the set $\Theta_\ell$ can be seen as a subset of $O(n+m,\F_\ell)$ consisting of matrices with fixed determinant (equal to $1$) and fixed spinor norm.

 \par\medskip
 Whereas the estimate~(\ref{grandO}) will be derived directly, for suitable families $(\Theta_\ell)$ in the case where $G=GL(n,\Z[1/\Pp])$, from~\cite[Appendix B]{KoLS}, the analogue study for $G=SO(n,m)(\Z)$ requires a few additional computations based on the ideas of~\cite[Lemmas $7.1$ and $7.2$]{KoCrelle}. The strategy of loc. cit. relies on the fact that as soon as a few particular conjugacy classes of $W_{N_{red}}$ (where we keep the notation $N_{red}=2\lfloor(n+m)/2\rfloor$) are detected in the Galois group of a polynomial $P\in\Z[T]$, the Galois group of $P$ is necessarily isomorphic to the whole group $W_{N_{red}}$. It is well-known that such conjugacy classes can be detected through the study of the factorisation patterns of $P\,(\mathrm{mod}\,\ell)$ with $\ell$ taking many different prime values. As explained in~\cite[Lemmas 7.1 and 7.3]{KoLS}, it is enough to consider four distinct families $(\Theta_\ell)$ or equivalently (via~(\ref{choix_enscriblants})) four families of sets of polynomials $(\tilde{\Theta}_\ell)$ (for simplicity we will denote in the sequel $N=n+m$ and $N_{red}=2\lfloor(n+m)/2\rfloor$):
 \begin{enumerate}
 \item Let $\tilde{\Theta}_\ell^{(1)}$ be the set of polynomials $f$ in $M_{N_{red},\ell}$ 
             \begin{itemize}\item which are irreducible \emph{if N is odd} or which are irreducible with a fixed value modulo nonzero squares of $\F_\ell$ in $-1$ and satisfy $\disc(f)=\disc(Q)$ \emph{if $N$ is even and $O(N_{red},\F_\ell)=O(N,\F_\ell)$ is nonsplit},
                             \item which factor as a product of two distinct monic irreducible polynomials of degree $N_{red}/2$ \emph{if $N=N_{red}$ is even, $O(N_{red},\F_\ell)$ is split and $\ell\equiv 1\,(\mathrm{mod}\,4)$},
                             \item which factor as a product of an irreducible monic quadratic polynomial and an irreducible polynomial of degree $N_{red}-2$ \emph{if $N=N_{red}$ is even, $O(N_{red},\F_\ell)$ is split and $\ell\equiv 3\,(\mathrm{mod}\,4)$}.    
              \end{itemize}
  \item Let $\tilde{\Theta}_\ell^{(2)}$ be the set of polynomials $f$ in $M_{N_{red},\ell}$ with a fixed value modulo nonzero squares of $\F_\ell$ in $-1$, which satisfy $\disc(f)=\disc(Q)$ and which factor as a product of a monic quadratic polynomial with distinct monic irreducible polynomials of odd degrees.
  \item Let $\tilde{\Theta}_\ell^{(3)}$ be the set of polynomials $f$ in $M_{N_{red},\ell}$ with a fixed value modulo nonzero squares of $\F_\ell$ in $-1$, which satisfy $\disc(f)=\disc(Q)$ and with associated polynomial $h$ (such that $f=x^nh(x+x^{-1})$) being separable with at least one factor of prime degree $>N_{red}/4$.   
  \item Let $\tilde{\Theta}_\ell^{(4)}$ be the set of polynomials $f$ in $M_{N_{red},\ell}$ with a fixed value modulo nonzero squares of $\F_\ell$ in $-1$, which satisfy $\disc(f)=\disc(Q)$ and with associated polynomial $h$ being separable with one irreducible quadratic factor and no other irreducible factor of even degree.                          
  \end{enumerate}

 \begin{lemme} \label{bon-theta}
 For $1\leqslant i\leqslant 4$ we have, with the above notation,
 $$
 \frac{|\Theta_{\ell}^{(i)}|}{|\Omega(N,\F_\ell)|}\gg 1\,,
 $$
 with an implied constant depending only on $N$.
 \end{lemme}
 
 \begin{proof}
 For $(\Theta_\ell^{(1)})_\ell$, let us first consider the case where $N=N_{red}$ is even. It is enough in the nonsplit case to combine Lemma~\ref{adaptcarlitz} and Lemma~\ref{lem7.2ko1}. This yields:
 $$
 \frac{|\Theta_{\ell}^{(1)}|}{|\Omega(N,\F_\ell)|}\geqslant \frac{1}{2N}\Bigl(1-\frac{2(1+N)}{\ell}\Bigr)\Bigl(1-\frac{1}{\ell+1}\Bigr)^{N(N-1)/2}\,,
 $$
 from which we derive the estimate we want.
 \par
  In the split case, the estimate we need is exactly the statement of the lemmas $6.5$ and $6.6$ of~\cite{Ka}. Indeed, we deduce directly from loc. cit.
  $$
  \frac{|\Theta_{\ell}^{(1)}|}{|\Omega(N,\F_\ell)|}\geqslant \frac{1}{4N^2}\,.
  $$
  
 Now if $N$ is odd, we invoke a new lemma of~\cite[Section 6]{Ka} (namely Lemma $6.4$ of loc. cit.) from which we get:
 $$
  \frac{|\Theta_{\ell}^{(1)}|}{|\Omega(N,\F_\ell)|}\geqslant \frac{1}{2N-2}\,,
  $$
for $\ell\geqslant \max(7, (N-1)/2)$.

\par
\medskip 
  As far as $(\Theta_\ell^{(2)})_\ell$ is concerned, we argue as in the proof of Lemma~\ref{lem7.2ko1} (when we embedded~(\ref{hs}) in $\Theta_\ell$) so that we can work with a quadratic space $V_{red}$ of dimension $N_{red}$ having the same discriminant as the ambiant $N$-dimenional space. Then, notice that we can take as a model of our quadratic space
  $$
  V_{red}=(\F_\ell^2, nonsplit)\oplus (\F_\ell^{N_{red}-2}, split)\,,
  $$ 
  if $O(N_{red},\F_\ell)$ is the nonsplit model for the orthogonal group in dimension $N_{red}$ over $\F_\ell$, and
  $$
  V_{red}=(\F_\ell^2, nonsplit)\oplus (\F_\ell^{N_{red}-2}, nonsplit)\,,
  $$
  if we deal with the split model for the orthogonal group.
  \par
   Now we proceed as in the proof of~\cite[Lemma 7.3]{KoLS} i.e. we consider separately the case where $N_{red}/2$ is even and the case where $N_{red}/2$ is odd. We perform the same computation as in loc. cit. with the slight difference that the (only) irreducible quadratic factor of each of the polynomial we consider takes imposed values modulo squares in $1$ and $-1$, so that we need to invoke Lemma~\ref{adaptcarlitz} to deduce the number of possible quadratic factors is greater or equal to $(\ell/4)\times(1-6/\ell)$ (instead of $(\ell/2)(1-1/\ell)$ in~\cite[Proof of Lemma $7.3(iii)$]{KoLS}). So combining loc. cit. and Lemma~\ref{lem7.2ko1}, we get
  \begin{align*}
  \frac{|\Theta_\ell^{(2)}|}{|\Omega(2n,\F_\ell)|}&\geqslant \frac{1}{4N_{red}}\Bigl(1-\frac{1}{\ell}\Bigr)^{N_{red}/2-1}\Bigl(1-\frac{1}{6\ell}\Bigr)\Bigl(1-\frac{1}{\ell+1}\Bigr)^{N_{red}(N_{red}-1)/2}\\
  &\geqslant \frac{1}{4N}\Bigl(1-\frac{1}{\ell}\Bigr)^{N/2-1}\Bigl(1-\frac{1}{6\ell}\Bigr)\Bigl(1-\frac{1}{\ell+1}\Bigr)^{N(N-1)/2}\,.
   \end{align*}
   hence the estimate we want to prove.
   
   \par
   \medskip
    Finally, for $(\Theta_\ell^{(3)})_\ell$ and $(\Theta_\ell^{(4)})_\ell$, it is straightforward to verify that the combinaton of Lemma~\ref{om-tilde-3-4} and Lemma~\ref{lem7.2ko1} yields the estimate of Lemma~\ref{bon-theta}.
 \end{proof}

   
 
  
   
   \par
   \medskip
    Next we turn to the case of the conjugacy coset sieve for $\alpha SL(n,\Z[1/\mt{P}])$. In that case, we have, for each $\ell\not\in\mt{P}$, $G_\ell^g=SL(n,\F_\ell)$ and the sieving sets we choose are still given by~(\ref{choix_enscriblants}) with this time, for any conjugacy class $c\in \mathfrak{S}_n$ whose elements have a decomposition in disjoint cycles involving $n_i$ cycles of length $i$ for $1\leqslant i\leqslant r$ ,
    $$
  \tilde{\Theta}_{\ell,c}=\{f\in \F_\ell[T]\mid f\,\text{has factorisation type } c\text{ and }\, f(0)=\det(\rho_\ell(\alpha))\}\,,
  $$
  where we say that a monic separable polynomial $f\in\F_\ell[T]$ of degree $r\geqslant 1$ has factorisation type $c\in\mathfrak{S}_r$ if $f$ factors as
  $$
  f=f_1\cdots f_r\,,
  $$
  where $f_i$ is a product of $n_i$ dictinct irreducible monic polynomials of degree $i$ and $\sum{i n_i}=r$.
  \par
  In the case of the trivial left coset $\rho_\ell(\alpha)=1$, the estimate we need is given by Kowalski in~\cite[Appendix B, Lemmas B.2 and B.5]{KoLS}. For the case of the general left coset, it is straightforward (by performing the obvious change of variable sending $\rho_\ell(\alpha)$ to $1$) to verify that the same estimate holds, so that we get the following result:
  
  \begin{lemme} \label{SL-estim}
  With the above notation, we have
  $$
 \frac{|\Theta_{\ell,c}|}{|SL(n,\F_\ell)|}\gg 1\,,
  $$
  as soon as $\ell>16n^2$, with an implied constant depending only on $n$.
  \end{lemme}
  
   Note that, to perform our sieve, the above statement suggests we should remove the primes smaller than $16n^2$ from $\Lambda$ but this does not affect the final result as the Dirichlet density of $\Lambda$ remains unchanged.
  
  \par
  \medskip
  Now we use the inequality:
  $$
  {\bf P}(\det(T-X_k)_{red}\,\text{does not have maximal Galois group})\leqslant\sum{{\bf P}\big(\Gal(\det(T-X_k))\cap \Theta^\sharp=\emptyset\big)}\,,
  $$
  where $\Theta^\sharp$ is the conjugacy class of $W_{N_{red}}$ (resp. $\mathfrak{S}_n$) determined by the family $\Theta=(\Theta_\ell)_\ell$ and where the sum runs over the family $(\Theta^{(i)})_{1\leqslant i\leqslant 4}$ (resp. $(\Theta_c)_{c\in\mathfrak{S}_n^\sharp}$) if $G=SO(n,m)(\Z)$ and $N_{red}=2\lfloor (n+m)/2\rfloor$ (resp. $G=GL(n,\Z[1/\Pp])$.
  \par Looking back once more at~(\ref{inegfond}) and applying both Proposition~\ref{borneconstconj} and the Prime Number Theorem, we obtain, for the two types of groups investigated
  $$
  {\bf P}(\det(T-X_k)_{red}\,\text{does not have maximal Galois group})\ll (1+L^{(3d+2)/2})\exp(-\eta k) L^{-1}\log L\,,
  $$
  with an implied constant depending on $n$ and the (strictly positive) density of $\Lambda$ in the set of all rational primes. If we set $L=\exp(\frac{2k\eta}{3d+2})$, then choosing for $\beta_3$ any positive real number smaller than $\frac{2\eta}{3d+2}$ yields $(i)$ of Theorem~\ref{exp_decrease_gen}.

 \subsection{Proof of Theorem~\ref{exp_decrease}}
 As a conclusion, we explain how to deduce Theorem~\ref{exp_decrease} from Theorem~\ref{exp_decrease_gen}. Notice first that the first part of Theorem~\ref{exp_decrease} is a trivial consequence of $(i)$ of Theorem~\ref{exp_decrease_gen}. As explained in the previous subsection, the only thing we need to prove to get the second inequality of Theorem~\ref{exp_decrease} is that the first order formula
 $$
 \varphi(x): \bigvee_{1\leqslant i,j\leqslant N}\exists y, y^2=x_{i,j}\,,
 $$
 yields sets $\Theta_\ell=\big(\neg\varphi(\F_\ell)\big)\cap Y_\ell$ (indexed by a set of primes $\Lambda_\delta$ to be determined) such that $|\Theta_\ell|\cdot|G_\ell^g|^{-1}\gg 1$.
 
 \par
  For both the case where $G=SO(n,m)(\Z)$ and $G=SL(n,\Z[1/\mt{P}])$, the above sets $\Theta_\ell$ can be expressed in the following way
 $$
 \Theta_\ell=\{g=(g_{i,j})\in \alpha_\ell G_\ell^g\mid g_{i,j}\,\text{is not a quare in}\,A/(\ell)=\F_\ell\}\,, 
 $$
 with notation as in Theorem~\ref{exp_decrease} and where $G_\ell^g$ denotes $\Omega(n,m)(\F_\ell)$ or $SL(n,\F_\ell)$ depending on what is $G$. The element $\alpha_\ell=\rho_\ell(\alpha)$ determines, in both cases, the left coset of $G_\ell$ (with respect to $G_\ell^g$) which contains $\Theta_\ell$.
\par
   We first consider the case where $G=SL(n,\Z[1/\mt{P}])$ (so $G_\ell^g=SL(n,\F_\ell)$ for $\ell\not\in\mt{P}$). As a representative of the left coset of $GL(n,\F_\ell)$ consisting of the matrices with fixed determinant say $d_\ell$, we can choose the diagonal matrix $\alpha_\ell$ given by $\alpha_{\ell,(1,1)}=d_\ell$ and $\alpha_{\ell,(i,i)}=1$ if $i\geqslant 2$. Using the Legendre character $(\frac{\cdot}{\ell})$ to detect squares, we want to evaluate
$$
\frac{1}{2|SL(n,\F_\ell)|}\sum_{\stacksum{g\in \alpha_\ell SL(n,\F_\ell)}{g_{i,j}\not =0}}{\Bigl(1+\Bigl(\frac{g_{i,j}}{\ell}\Bigr)\Bigr)}\,.
$$

 Thus to obtain the inequality $|\Theta_\ell||SL(n,\F_\ell)|^{-1}\gg 1$, it is enough to prove
$$
\sum_{g\in (\alpha_\ell SL(n,\F_\ell))_{i,j}}{\Bigl(\frac{g_{i,j}}{\ell}\Bigr)}\ll \ell^{d-1/2}\,,
$$ 
where we denote $(\alpha_\ell SL(n,\F_\ell))_{i,j}$ the matrices of $\alpha_\ell SL(n,\F_\ell)$ with a nonzero entry in position $(i,j)$, and where $d$ equals $n^2-1$, the dimension of the algebraic group ${\bf SL}(n)$.
\par
 Now, for each $g\in \alpha_\ell SL(n,\F_\ell)$, there exists $h\in SL(n,\F_\ell)$ such that $g=\alpha_\ell h$. The $(i,j)$-th entry of $g$ is given by
$$
g_{i,j}=\sum_k{\alpha_{\ell,(i,k)}h_{k,j}}=\alpha_{\ell,(i,i)}h_{i,j}\,,
$$ 
since the matrix $\alpha_\ell$ is diagonal. So it is enough to prove	
$$
\Bigl(\frac{\alpha_{\ell,(i,i)}}{\ell}\Bigr) \sum_{h\in SL(n,\F_\ell)_{i,j}}{\Bigl(\frac{h_{i,j}}{\ell}\Bigr)}\ll \ell^{d-1/2}\,,
$$
which can also be written in the following way:
$$
\sum_{h\in SL(n,\F_\ell)_{i,j}}{\Bigl(\frac{h_{i,j}}{\ell}\Bigr)}\ll \ell^{d-1/2}\,.
$$

 This inequality is proved in~\cite[Appendix B., Prop. $B.4$]{KoLS}. Thus~(\ref{cond}) of Theorem~\ref{exp_decrease_gen} $(ii)$ holds if we choose for $\Lambda_\delta$ the complementary of $\Pp$ in the rational primes, and we deduce the second part of Theorem~\ref{exp_decrease} in the case where $G=GL(n,\Z[1/\Pp])$.

\par\medskip
 Finally, in the case where $G=SO(n,m)(\Z)$ (i.e. $G_\ell^g=\Omega(n,m)(\F_\ell)$), things are slightly different as this time, the group $\Omega(n,m)(\F_\ell)$ cannot be seen as the group of $\F_\ell$-points of an algebraic group. In order to end up applying the same techniques as above, we need to relate for fixed indices $1\leqslant i,j\leqslant n+m$ the cardinality of the sieving set
 $$
 \Theta_\ell=\{g\in \alpha_\ell \Omega(n,m)(\F_\ell)\mid g_{i,j}\,\text{is a square in}\, \F_\ell \}\,,
 $$
to the cardinality of
$$
\Theta^{SO}_\ell=\{g\in SO(n,m)(\F_\ell)\mid g_{i,j}\,\text{is a square in}\, \F_\ell \}\,.
$$

 To that purpose, we first make a special choice for the basis thanks to which we identify the elements of $SO(n,m)(\F_\ell)$ with their matrix representation. Indeed the surjectivity of the spinor norm (see Section~\ref{localdensities}) onto $\{\pm 1\}$ enables us to choose a vector, say $e_1$, such that $\nsp(r_{e_1})=-1$, where $r_{e_1}$ denotes the reflection with respect to the hyperplane $\mt{H}_{e_1}$ which is orthogonal to $e_1$. We can consider the restriction of the quadratic form $Q$ on the ambiant space to $\mt{H}_{e_1}$. This restriction is still non degenerate~(\cite[page 139]{OM}) and we can choose the second vector $e_2$ of the basis we are constructing in such a way that the corresponding reflection $r^{\mt{H}_{e_1}}_{e_2}$ of $\mt{H}_{e_1}$ has spinor norm $1$ (see~\cite[Proof of Lemma $6.3$]{Ka}). Then we can complete the basis of $\mt{H}_{e_1}$ with vectors $\{e_3,\ldots,e_{n+m}\}$ in such a way that the matrix of $r^{\mt{H}_{e_1}}_{e_2}$ written in the basis $(e_2,\ldots,e_{n+m})$ has diagonal coefficients $(-1,1,\ldots,1)$ and zeros outside the diagonal. Finally we can extend $r^{\mt{H}}_{e_2}$ to the whole ambiant space by making it act trivially on $e_1$. We get a reflexion $r_{e_2}$ of the whole space. Now the product $r_{e_1}r_{e_2}$ is an element of $SO(n,m)(\F_\ell)$ with spinor norm $-1$ (note that $\nsp(r_{e_2})=\nsp(r^{\mt{H}_{e_1}}_{e_2})$) and the matrix $M_{e_1,e_2}$ of $r_{e_1}r_{e_2}$ in the basis $(e_i)_i$ is the diagonal matrix with $M_{e_1,e_2}(i,i)=-1$ if $i=1,2$ and $M_{e_1,e_2}(i,i)=1$ otherwise. Now consider the involution
 \begin{align*}
 \Theta_\ell&\ra \Theta_\ell^\eps\\
 g&\mapsto M_{e_1,e_2}g 
 \end{align*} 
 where $\eps$ is any representative of the left coset of $SO(n,m)(\F_\ell)$ consisting of elements with spinor norm $-1$ and
 $$
 \Theta_\ell^\eps=\{g\in (\eps\alpha_\ell) \Omega(n,m)(\F_\ell)\mid g_{i,j}\,\text{is a square in}\, \F_\ell \}\,.
 $$
 
  We have $(M_{e_1e_2}g)_{i,j}=-g_{i,j}$ if $j=1,2$ and $(M_{e_1e_2}g)_{i,j}=g_{i,j}$ otherwise. So, for primes $\ell\equiv 1\,(\mathrm{mod}\,4)$, $g_{i,j}$ is a square in $\F_\ell$ if and only if $(M_{e_1e_2}g)_{i,j}$ is a square as well. For such primes, we deduce that, 
  $$
  |\Theta^{SO}_\ell|=|\Theta_\ell|+|\Theta_\ell^\eps|=2|\Theta_\ell|\,.
  $$ 
 
 So
 \begin{align*}
\frac{1}{2|\Omega(n,m)(\F_\ell)|}\sum_{\stacksum{g\in \alpha_\ell \Omega(n,m)(\F_\ell)}{g_{i,j}\not =0}}{\Bigl(1+\Bigl(\frac{g_{i,j}}{\ell}\Bigr)\Bigr)}
&=\frac{1}{4|\Omega(n,m)(\F_\ell)|}\sum_{\stacksum{g\in SO(n,m)(\F_\ell)}{g_{i,j}\not =0}}{\Bigl(1+\Bigl(\frac{g_{i,j}}{\ell}\Bigr)\Bigr)}\\
&=\frac{1}{2|SO(n,m)(\F_\ell)|}\sum_{\stacksum{g\in SO(n,m)(\F_\ell)}{g_{i,j}\not =0}}{\Bigl(1+\Bigl(\frac{g_{i,j}}{\ell}\Bigr)\Bigr)}\,,\\
\end{align*}
as soon as we restrict ourselves to primes such that $\ell\equiv 1\,(\mathrm{mod}\,4)$.

\par
 To deduce from the last inequality that $|\Theta_\ell||\Omega(n,m)(\F_\ell)|^{-1}\gg 1$ (with an implied constant depending only on $n$ and $m$), we use the exact same argument as in the previous case (where the finite group involved was $SL(n,\F_\ell)$). Indeed, we can now see the sum investigated as a character sum over the $\F_\ell$-points of the geometrically irreducible variety ${\bf SO}(n,m)/\F_\ell$.
 
 \par
 \medskip
 So we can choose $\Lambda_\delta=\{\text{primes congruent to } 1\text{ modulo }4\}$ (which has Dirichlet density $1/2$) and then apply once more $(ii)$ of Theorem~\ref{exp_decrease_gen} to get the second part of Theorem~\ref{exp_decrease} in the case $G=SO(n,m)(\Z)$. 
  

\newpage

\section{Appendix: Coset Sieves} 

\par
\bigskip

 
 The purpose of this appendix is to explain the role that the large sieve plays in the proof of Theorem~\ref{exp_decrease_gen}. We give here the full statements with proofs of the different kinds of a priori estimates we need to get the kind of explicit upper bounds of Theorem~\ref{exp_decrease_gen}. The results we expose here are very much in the spirit of~\cite{KoLS} (especially Section $3.3$ of loc. cit.) and sometimes, we only recall some results of~\cite{KoLS}. Moreover in the last subsection, we give self-contained versions of these statements in order to make it possible for the reader to follow the proof of Theorem~\ref{exp_decrease_gen} without having to get too much involved in the details of the sieving machinery.
 
 \subsection{The general framework}
 Our general sieving context is that of the \emph{coset sieve}. A general description of that sieve goes as follows: we suppose we are given a discrete group $G$ with a normal subgroup $G^g$ such that the quotient $\Gamma=G/G^g$ is abelian. Moreover, we suppose that there exists a subset $\Lambda$ of the rational primes such that, for any $\ell\in\Lambda$, we have a \emph{surjective} group homomorphism
 $$
 \rho_\ell: G\ra G_\ell\,, 
 $$
 where $G_\ell$ is a \emph{finite} group.
 \par
  That is, of course, a very natural generalisation of the reduction modulo $\ell$ morphism from $\Z$ to $\Z/\ell\Z$. We emphasize here the fact that the above data is really all we need to set the sieve for cosets that we apply (and this is really the strong idea underlying~\cite[Chap. 3.3]{KoLS}). All the framework we build from there, comes from ``natural'' deductions. First, we denote
  $$
  G_\ell^g=\rho_\ell(G^g)\,
  $$ 
  for $\ell\in\Lambda$. That subgroup is normal in $G_\ell$ because $G^g$ is a normal subgroup of $G$ and, since for every $\ell\in\Lambda$, the morphism $\rho_\ell$ is onto. Now, all these groups fit the following commutative diagram with exact rows (such a diagram can already be found, in a geometric context, in~\cite[Th. 4.1]{Ch}, and is extensively used in~\cite{KoCrelle}):
   \begin{equation} \label{diagcosetsieve}
 \begin{CD}
1 @>>> G^g @>>> G @>d>> \Gamma=G/G^g @>>> 1\\
@. @VV \rho_\ell V @VV \rho_\ell V @VV{\rm pr}_\ell V\\
1 @>>> G_\ell^g @>>> G_\ell @> d >> \Gamma_\ell=G_\ell/G_\ell^g @>>> 1
\end{CD}
 \end{equation}  
 where the surjective morphism ${\rm pr}_\ell$ is defined in such a way that the diagram commutes. In both rows, the quotient map is denoted $d$, in order to avoid the introduction of an additional notation.
 
 \par
 \medskip
  An important feature of that sieve setting, in view of the proof of $(i)$ of Theorem~\ref{exp_decrease_gen}, is that each left coset $\alpha G^g$ of $G$ ($\alpha$ being any fixed element of $G$) is conjugacy invariant. This comes from the fact that the quotient $\Gamma=G/G^g$ is an abelian group. 
  \par
  We now fix an element $\alpha\in G$, and, use, from now on, the upper index $^\sharp$ to denote the set of conjugacy classes of the (conjugacy invariant) set considered. The two following sieve settings $(Y,\Lambda,(\rho_\ell: Y\ra Y_\ell))$ will be useful for our purpose:
  \begin{itemize}
  \item either $Y=(\alpha G^g)^\sharp$, $Y_\ell=(\rho_\ell(\alpha) G_\ell^g)^\sharp$ and $\rho_\ell$ also denotes the restriction (which remains surjective) $Y\ra Y_\ell$, for any $\ell\in \Lambda$ (this will be referred to as the \emph{conjugacy coset sieve}).
  \item or $Y=\alpha G^g$, $Y_\ell=\rho_\ell(\alpha) G_\ell^g$ and $\rho_\ell$ also denotes the (surjective) restriction $Y\ra Y_\ell$, for any $\ell\in \Lambda$ (this will be referred to as the \emph{non-conjugacy coset sieve}).
  \end{itemize}

  As we do in the introduction, we assume we are given a probability space $(\Psi,\Sigma,{\bf P})$. The random walk $(X_k)$ we are interested in (see the introduction again), can be seen as an application on $\Psi$ with values in $Y$ (whatever choice we make for the set $Y$ among the two possibilities above). For each $k$, we end up with a \emph{siftable set} $(\Psi,X_k,{\bf P})$. Following Kowalski's book~\cite{KoLS}, let us denote by $\mt{L}^*$ (the \emph{prime sieve support})  the set of elements $\ell\in\Lambda$ that are smaller than a fixed integer $L\geqslant 1$. The large sieve method we use here consists in giving, for any family $\Theta=(\Theta_\ell)$ of subsets of $Y_\ell$ (the \emph{sieving sets}) indexed by $\Lambda$, an upper bound for the probability
  $$
  {\bf P}(\{x\in \Psi\mid \rho_\ell(X_k(x))\not\in\Theta_\ell\,\,\text{for all}\, \ell\in\Lstar\})\,.
  $$
  
   For a matter of convenience, we will rewrite this probability in the standard way
  $$
  {\bf P}(\rho_\ell(X_k)\not\in\Theta_\ell\,\,\text{for all}\, \ell\in\Lstar)\,.
  $$
  
   \par
   We need to make precise the meaning and the definition of the constants $\Delta$ and $H$ appearing in the fundamental inequality~\ref{inegfond} (the constant $\Delta$ is sometimes denoted $\Delta(X_k,\Lstar)$ when we want to emphasize the dependency on the parameters). From now on, we will assume we are given a probability density $\nu_\ell$ on $\Theta_\ell$ for any $\ell\in\Lambda$; then the constant $H$ can be taken to be equal to
   $$
   H=\sum_{\ell\in\Lstar}{\frac{\nu_\ell(\Theta_\ell)}{1-\nu_\ell(\Theta_\ell)}}\,.
   $$ 
   
   To define the large sieve constant $\Delta(X_k,\Lstar)$, we should first emphasize that the central issue, in order to get a useful upper bound from~(\ref{inegfond}), is to find a suitable basis for the space $L^2(Y_\ell,\nu_\ell)$ which is the complex Hilbert space with associated inner product (defined for $\C$-valued functions $f$ and $g$ on $Y_\ell$):
   $$
   \langle f;g\rangle=\sum_{y\in Y_\ell}{\nu_\ell(y)f(y)\overline{g(y)}}\,.
   $$
   
   If $\mt{B}_\ell$ denotes an orthonormal basis of that space containing the constant function $1$ and if $\mt{B}_\ell^*=\mt{B}_\ell\setminus\{1\}$, then, for any square integrable function $\beta: \Psi\ra \C$, the large sieve constant $\Delta(X_k,\Lstar)$ is defined as the smallest constant $\Delta$ satisfying
   $$
   \sum_{\ell\in \Lstar}\sum_{\varphi\in\mt{B}_\ell^*}{\Bigl|\int_\Psi\beta(\omega)\varphi(\rho_\ell(X_k(\omega)))d{\bf P}(\omega)\Bigr|^2}\leqslant \Delta\int_\Psi|\beta(\omega)|^2d{\bf P}(\omega)\,,
   $$
   which can also be written, denoting by ${\bf E}(X)$ the expectation of a random variable $X$,
  $$
   \sum_{\ell\in \Lstar}\sum_{\varphi\in\mt{B}_\ell^*}{\Bigl|{\bf E}(\beta\cdot\varphi(\rho_\ell(X_k)))\Bigr|^2}\leqslant \Delta{\bf E}(|\beta|^2)\,.
   $$ 

The proof of the inequality~(\ref{inegfond}) as we state it, can be found in~\cite[Prop. 2.3]{KoLS}. To find an upper bound for $\Delta(X_k,\Lstar)$, we use (see~\cite[Section 5]{KoCrelle} for an analogue given in a geometric context):
\begin{equation} \label{majdelta}
\Delta\leqslant \max_{\ell\in\Lstar}\max_{\varphi'\in\mt{B}_\ell^*}{\sum_{\ell'\in\Lstar}\sum_{\varphi\in \mt{B}_\ell^*}|W(\varphi,\varphi')|}\,,
\end{equation}    
 where the $W(\varphi,\varphi')$ are the ``exponential sums'' given by
\begin{equation}\label{expW}
 W(\varphi,\varphi')={\bf E}(\varphi(\rho_{\ell'}(X_k))\overline{\varphi'(\rho_\ell(X_k))})\,.
\end{equation}

  Obviously the usefulness of~(\ref{majdelta}) lies in the fact that it now suffices to give estimates for the individual sums $W(\varphi,\varphi')$ to deduce an upper bound for the large sieve constant. In our context where the sets $Y_\ell$ are left cosets in finite groups, it is natural to use the irreducible characters of the  finite groups $G_\ell$ in order to construct a suitable basis $\mt{B}_\ell$. Moreover, that explains why it seems fair to call the $W(\varphi,\varphi')$ exponential sums.

  \subsection{Exhibiting orthonormal bases}
  In what follows, we give the description, for each $\ell\in\Lambda$, of the basis $\mt{B}_\ell$ we need (note that we need to handle both the conjugacy coset sieve and the non-conjugacy coset sieve). 
  \par
  First, we recall the following Lemma, due to Kowalski (see~\cite[Lemma 3.2]{KoLS}) in which a basis for $L^2(Y_\ell,\nu_\ell)$ is described in the case of a conjugacy coset sieve (we recall that in that case, $Y_\ell=(\rho_\ell(\alpha) G_\ell^g)^\sharp$) where the density $\nu_\ell$ is the uniform density defined by $\nu_\ell(y^\sharp)=|y^\sharp||G_\ell^g|^{-1}$.
  
  \begin{lemme} \label{baseconjsieve}
With the same notation as above, let 
  $$
  \varphi_\pi: y^\sharp\in G_\ell^\sharp\mapsto \Tr(\pi(y^\sharp))\,,
$$   
where $\ell\in\Lambda$ and $\pi$ is an irreducible representation of $G_\ell$. Let $L^2(Y_\ell,\nu_\ell)$ denote the Hilbert space of complex valued square integrable functions with respect to the scalar product:
 $$
 \langle f;g\rangle=\frac{1}{|G_\ell^g|}\sum_{y^\sharp\in Y_\ell}{\nu_\ell(y^\sharp)f(y^\sharp)\overline{g(y^\sharp)}}\,. 
 $$

Then we have
 \begin{enumerate}
 \item If $\pi$ and $\tau$ are irreducible representations of $G_\ell$, 
 \begin{equation} \label{prod_phipi_phitau}
   \langle\varphi_\pi;\varphi_\tau\rangle=\left\{\begin{array}{l} 0\,, \,\text{ if }\,\pi_{|G_\ell^g}\not\simeq\tau_{|G_\ell^g}  \,\text{ or }\, {\varphi_\pi}_{|Y_\ell}=0\,,\\ \\\overline{\psi(d\circ\rho_\ell(\alpha))}|\hat{\Gamma}_\ell^\pi|\,,\text{ otherwise }\,, \end{array}\right.
  \end{equation}
 where $\hat{\Gamma}_\ell$ is the character group of $\Gamma_\ell$ and $\psi$ is an element of the group $\hat{\Gamma}_\ell^\pi=\{\psi\in\hat{\Gamma}_\ell\mid \pi\simeq \pi\otimes \psi\}$. 
  \item Let $\mt{B}_\ell$ be the family of functions
  \begin{align*}
  Y_\ell &\ra \C \\
  y^\sharp &\mapsto |\hat{\Gamma}_\ell^\pi|^{-1/2}\varphi_\pi(y^\sharp)\,,
  \end{align*}
  where $\pi$ runs over the subset $\Pi_\ell^*$ of a set $\Pi_\ell$ of representatives for the irreducible representations of $G_\ell$ with respect to the equivalence relation
  $$
  \pi\sim\tau\,\, \text{if and only if}\,\, \pi_{|G_\ell^g}\simeq \tau_{|G_\ell^g}\,,
  $$
  and where $\pi\in\Pi_\ell^*$ if and only if ${\varphi_\pi}_{|Y_\ell}\not=0$. Then the family $\mt{B}_\ell$ is an orthonormal basis for $L^2(Y_\ell,\nu_\ell)$.
 \end{enumerate}
\end{lemme}

 In the case of the non-conjugacy coset sieve, the irreducible representations of $G_\ell$ cannot be used in such a direct way to construct $\mt{B_\ell}$, but in spite of that, they turn out to be very useful once more.
 \par
 Fix an $\ell\in\Lambda$ and a finite dimensional irreducible representation $\pi_\ell$ of $G_\ell$. Let
 $$
B_{\pi_\ell}=(e_{\pi_\ell}^1,\ldots,e_{\pi_\ell}^{\dim \pi_\ell})
$$
denote an orthonormal basis of the representation space $V_{\pi_\ell}$ of $\pi_\ell$, with respect to a $G_\ell$-invariant inner product on $V_{\pi_\ell}$ denoted $\langle\,;\,\rangle_{\pi_\ell}$ (note that, $\rho_\ell$ having finite image, it is always possible to assume the existence of such a $G_\ell$-invariant inner product). Then for any two elements $e$ and $f$ of $\{e_{\pi_\ell}^1,\ldots,e_{\pi_\ell}^{\dim \pi_\ell}\}$, consider the function
$$
 \varphi_{\pi_\ell,e,f}: x\in G_\ell\mapsto \sqrt{\dim \pi_\ell} \langle\pi_\ell(x)e;f\rangle_{\pi_\ell}\,,
 $$
 called a \emph{matrix coefficient}.
 
 \par
  Then the family $(\varphi_{{\pi_\ell},e,f})$ obtained by varying $\pi_\ell$ in $\Pi_\ell$ (where $\Pi_\ell$ denotes a set of respresentatives for the isomorphism classes of irreducible representations of $G_\ell$) and $e, f$ in a fixed basis $B_{\pi_\ell}$, forms an orthonormal basis for $L^2(G_\ell,\nu_\ell)$ of square integrable complex valued functions on $G_\ell$ with respect to the inner product
  $$
  \langle f;g\rangle=\frac{1}{|G_\ell|}\sum_{x\in G_\ell}{f(x)\overline{g(x)}}\,,
  $$
corresponding to the uniform density $\nu_\ell$ defined for $y\in G_\ell$ by $\nu_\ell(y)=|G_\ell|^{-1}$ (see~\cite[Section I.5]{Kn} for a proof).
 \par
 From that result we can derive, in the non-conjugacy sieve setting, a useful orthonormal basis for $L^2(Y_\ell,\nu_\ell)$, where we recall that $Y_\ell=\rho_\ell(\alpha)G_\ell^g$ and where, for $y\in Y_\ell$, $\nu_\ell(y)=|G_\ell^g|^{-1}$:

 \begin{lemme} \label{basenonconjsieve}
 With notation as above, consider the inner product on $L^2(Y_\ell,\nu_\ell)$ defined by 
 $$
 \langle f;g\rangle=\frac{1}{|G_\ell^g|}\sum_{y\in Y_\ell}{f(y)\overline{g(y)}}\,,
 $$
 
  Let $\pi_\ell$ and $\tau_\ell$ be irreducible representations of $G_\ell$ and $(e,f)$ (resp. $(\epsilon,\phi)$) a couple of elements of an orthonormal basis $B_{\pi_\ell}$ (resp. $B_{\tau_\ell}$) of $V_{\pi_\ell}$ (resp. $V_{\tau_\ell}$). The functions $\varphi_{\pi_\ell,e,f}$ and $\varphi_{\tau_\ell,\epsilon,\phi}$ are said to be \emph{equivalent} (in which case we will note $\varphi_{\pi_\ell,e,f}\sim\varphi_{\tau_\ell,\epsilon,\phi}$) if the entry $(e,f)$ of ${\rm Mat}_{B_{\pi_\ell}}\pi_\ell(g)$ and the entry $(\epsilon,\phi)$ of ${\rm Mat}_{B_{\tau_\ell}}\tau_\ell(g)$ coincide for all $g\in G_\ell^g$. 
  \par
  Then 
 \begin{enumerate}
 \item If $\pi_\ell$ and $\tau_\ell$ are irreductible representations of $G_\ell$, and if we denote
 $$
  \hat{\Gamma}_\ell^{\varphi_{\pi_\ell,e,f}}=\{\chi\in\hat{\Gamma}_\ell\mid \varphi_{\pi_\ell,e,f}\otimes\chi=\varphi_{\pi_\ell,e,f}\,\text{in}\, L^2(G_\ell)\}\,,
  $$
  we have
$$
 \langle\varphi_{\pi_\ell,e,f};\varphi_{\tau_\ell,\epsilon,\phi}\rangle=\left\{\begin{array}{l}  0\,\,\text{if}\, \varphi_{\pi_\ell,e,f} \not\sim \varphi_{\tau_\ell,\epsilon,\phi}\,\text{or if the entry}\, (e,f)\, \text{(resp.} (\epsilon,\phi))\\ \,\,\,\,\text{of}\,\, {\rm Mat}_{B_{\pi_\ell}}\pi_\ell(g)\,\text{(resp.}\, {\rm Mat}_{B_{\tau_\ell}}\tau_\ell(g))\,\text{is zero for all}\,\, g\in Y_\ell\,, \\ \\
 \overline{\psi(d(\alpha_\ell))}|\hat{\Gamma}_\ell^{\varphi_{\pi_\ell,e,f}}|\,\text{otherwise, where}\, \alpha_\ell=\rho_\ell(\alpha),\, \psi\in\hat{\Gamma}_\ell\,\text{and}\, \varphi_{\pi_\ell,e,f}\otimes \psi\simeq \varphi_{\tau_\ell,\epsilon,\phi}\,. \end{array}\right.
 $$
 \item Let $\mt{B}_\ell$ be the family of functions
 \begin{align*}
 Y_\ell&\ra\ \C\\
 x & \mapsto |\hat{\Gamma}_\ell^{\varphi_{\pi_\ell,e,f}}|^{(-1/2)}\varphi_{\pi_\ell,e,f}(x)\,,
 \end{align*}
 where $(\pi_\ell,e,f)$ runs over the triples corresponding to a system of representatives for the equivalence relation $\sim$ and where we assume that for every triple $(\pi_\ell,e,f)$, there exists an element $g\in Y_\ell$ such that the entry $(e,f)$ of ${\rm Mat}_{B_{\pi_\ell}}\pi_\ell(g)$ is nonzero. Then $\mt{B}_\ell$ is an orthonormal basis for $L^2(Y_\ell,\nu_\ell)$.
 \end{enumerate}
 \end{lemme}

\begin{proof}
$(1)$ We evaluate the scalar product
 \begin{align*}
 \langle\varphi_{\pi_\ell,e,f};\varphi_{\tau_\ell,\epsilon,\phi}\rangle&=\frac{1}{|G_\ell^g|}\sum_{y\in Y_\ell}{\varphi_{\pi_\ell,e,f}(y)\overline{\varphi_{\tau_\ell,\epsilon,\phi}(y)}}\\
 &=\frac{1}{|G_\ell^g|}\sum_{\stacksum{y\in G_\ell}{d(y)=d(\alpha_\ell)}}{\varphi_{\pi_\ell,e,f}(y)\overline{\varphi_{\tau_\ell,\epsilon,\phi}(y)}}\\
 &=\frac{1}{|G_\ell^g|}\frac{1}{|\hat{\Gamma}_\ell|}\sum_{y\in G_\ell}{\Bigl(\sum_{\psi\in\hat{\Gamma}_\ell}{\overline{\psi(\alpha_\ell)}\psi(y)}\Bigr)\varphi_{\pi_\ell,e,f}(y)\overline{\varphi_{\tau_\ell,\epsilon,\phi}(y)}}\,,
 \end{align*}
 where the last inequality is obtained using Frobenius reciprocity.
 \par
 
 Obviously the right hand side of the above equality vanishes as soon as  ${\varphi_{\pi_\ell,e,f}}_{|Y_\ell}$ or ${\varphi_{\tau_\ell,\epsilon,\phi}}_{|Y_\ell}$ is identically zero (which corresponds respectively to the vanishing of the entry $(e,f)$ of ${\rm Mat}_{B_{\pi_\ell}}\pi_\ell(g)$ or $(\epsilon,\phi)$ of ${\rm Mat}_{B_{\tau_\ell}}\tau_\ell(g)$, for all $g\in Y_\ell$). However, if that quantity does not vanish, we have, on the right hand side,
 $$
 \sum_{\psi\in\hat{\Gamma}_\ell}{\overline{\psi(d(\alpha_\ell))}\langle\varphi_{\pi_\ell,e,f}\otimes\psi;\varphi_{\tau_\ell,\epsilon,\phi}\rangle_{G_\ell}}\,.
 $$
 
 Now, in $L^2(Y_\ell,\nu_\ell)$, we have the equality of functions $\varphi_{\pi_\ell,e,f}\otimes\psi=\varphi_{\pi_\ell\otimes\psi,e,f}$. Indeed any $G_\ell$-invariant scalar product $\langle\,;\,\rangle_{\pi_\ell}$ on $V_{\pi_\ell}\simeq V_{\pi_\ell\otimes\psi}$ (as vector spaces) remains $G_\ell$-invariant if $G_\ell$ acts via $\pi_\ell\otimes \psi$ (which is still an irreducible representation of $G_\ell$). We deduce
 $$ \langle\varphi_{\pi_\ell,e,f};\varphi_{\tau_\ell,\epsilon,\phi}\rangle=\sum_{\psi\in\hat{\Gamma}_\ell}{\overline{\psi(d(\alpha_\ell))}\delta(\varphi_{\pi_\ell\otimes\psi,e,f},\varphi_{\tau_\ell,\epsilon,\phi})}\,,
 $$
where $\delta$ denotes Kronecker's symbol.
 \par
  The quantity $\delta(\varphi_{\pi_\ell\otimes\psi,e,f},\varphi_{\tau_\ell,\epsilon,\phi})$ equals $1$ if and only if $\varphi_{\pi_\ell\otimes\psi,e,f}= \varphi_{\tau_\ell,\epsilon,\phi}$ in $L^2(G_\ell)$. This condition is equivalent to the coincidence of the restrictions ${\varphi_{\pi_\ell,e,f}}_{|G_\ell^g}$ and ${\varphi_{\tau_\ell,\epsilon,\phi}}_{|G_\ell^g}$, i.e. the equality between the entry $(e,f)$ of ${\rm Mat}_{B_{\pi_\ell}}\pi_\ell(g)$ and the entry $(\epsilon,\phi)$ of ${\rm Mat}_{B_{\tau_\ell}}\tau_\ell(g)$, for all $g\in G_\ell^g$. In that case we deduce, using~\cite[Lemma 3.2]{KoLS},
  $$
  \langle\varphi_{\pi_\ell,e,f};\varphi_{\tau_\ell,\epsilon,\phi}\rangle=\overline{\psi(d(\alpha_\ell))}|\hat{\Gamma}_\ell^{\varphi_{\pi_\ell,e,f}}|\,,
  $$
 where $\psi$ is any of the characters of $\hat{\Gamma}_\ell$ such that
  $$
  \varphi_{\pi_\ell,e,f}\otimes\psi\simeq \varphi_{\tau_\ell,\epsilon,\phi}\,.
  $$
  
  The assertion $(2)$ is straightforward using $(1)$ and the above arguments.
  
\end{proof}

\begin{remarks}
($i$) In the course of the proof of Lemma~\ref{basenonconjsieve}, we have seen that the relation $\varphi_{\pi_\ell,e,f}\sim\varphi_{\tau_\ell,\epsilon,\phi}$ is equivalent to the existence of a character $\psi\in\hat{\Gamma}_\ell$ such that
 $$
 \varphi_{\pi_\ell,e,f}\otimes \psi=\varphi_{\pi_\ell\otimes\psi,e,f}=\varphi_{\tau_\ell,\epsilon,\phi}\,,
 $$
 in $L^2(G_\ell)$. Then we proved that the scalar product
 $$
  \langle\varphi_{\pi_\ell,e,f};\varphi_{\tau_\ell,\epsilon,\phi}\rangle\,,
 $$
 is equal to $\overline{\psi(d(\alpha_\ell))}|\{\chi\in\hat{\Gamma}_\ell\mid \varphi_{\pi_\ell,e,f}\otimes\chi= \varphi_{\pi_\ell,e,f} \text{ in } L^2(G_\ell)\}|$. Thus, if $\pi_\ell$ is an irreducible representation of $G_\ell$, the group $\hat{\Gamma}^{\pi_\ell}_\ell$, in the sense of Lemma~\ref{baseconjsieve}, is a subgroup of $\hat{\Gamma}_\ell^{\varphi_{\pi_\ell,e,f}}$ for any choice of vectors $e,f$ in an orthonormal basis of the representation space $V_{\pi_\ell}$. Indeed, if $\psi\in \hat{\Gamma}_\ell^\pi$, we have an isomorphism of $G_\ell$-representations: $\pi_\ell\otimes\psi\simeq\pi_\ell$, hence,
 \begin{align*}
 (\varphi_{\pi_\ell,e,f}\otimes\psi)(g)&=\varphi_{\pi_\ell\otimes\psi,e,f}(g)=\langle\pi_\ell\otimes\psi(g)e;f\rangle_{\pi_\ell}\\
                                       &=\langle\pi_\ell(g)e;f\rangle_{\pi_\ell}=\varphi_{\pi_\ell,e,f}(g)\,,
 \end{align*}
 for every $g\in G_\ell$. That means that the equivalence relation of Lemma~\ref{baseconjsieve} is ``stronger'' than the one described in Lemma~\ref{basenonconjsieve} (in the sense that the classes for the former relation, which are contained in those for the latter, may form strict subsets in those classes).
 \par
($ii$) Using the example of the dihedral group $D_{n}$, $n\,\text{even}\,\geqslant 2$, we see that the equivalence relation $\sim$ defined in Lemma~\ref{basenonconjsieve} can actually be non trivial, i.e., we can exhibit two non isomorphic irreductible representations $(\pi,V_\pi)$ and $(\tau,V_\tau)$ of $D_n$ and two couples of vectors $(e,f)$ and $(\epsilon,\phi)$ (respectively in $V_\pi$ and $V_\tau$) such that $\varphi_{\pi,e,f}(g)=\varphi_{\tau,\epsilon,\phi}(g)$, for all $g\in G^g$, with a suitable choice of group $G^g$.
 \par
 With notation as in~\cite[$5.3$]{Se}, let $G^g=C_n$, be the cyclic group of order $n$. It is of index $2$ in $D_n$ and we have an exact sequence of finite groups:
 $$
 1\ra C_n\ra D_n\ra \Z/2\Z\ra 1\,.
 $$
 
  We fix the trivial left coset representative $\alpha=1$ (with respect to the quotient $D_n/C_n$). If $h_1$ and $h_2$ are two distinct integers such that $h_1\equiv -h_2\,(\mathrm{mod}\,n)$, then the representations $\rho^{h_1}$ and $\rho^{h_2}$ given in the canonical basis of $\C^2$ by
 $$
 \rho^{h_i}(r^k)=\Bigl(\begin{array}{cc}
  \omega^{h_ik} & 0\\ 0 & \omega^{-h_i k} 
  \end{array} \Bigr)\,,
 $$
 (where, as in loc. cit.,  $r$ is a generator for $C_n$, $0\leqslant k\leqslant n-1$ and $\omega=\exp(2i\pi/n)$), are irreducible of degree $2$ and are \emph{not isomorphic}. A straightforward computation shows that the canonical basis $(e,f)$ of $\C^2$ is in fact an orthonormal basis of the representation space $V_{\rho^{h_i}}$, $i=1,2$, with respect to the $D_n$-invariant inner product $ \langle\,;\,\rangle_{\rho^{h_i}}$ constructed from the canonical scalar product on $\C^2$ by averaging over $D_n$. Then we have, for all $g\in C_n$,
 $$
 \varphi_{\rho^{h_1},e,e}(g)=\varphi_{\rho^{h_2},f,f}(g)\,,
 $$
since, for all $0\leqslant k\leqslant n-1$, the equality $\omega^{h_1k}=\omega^{-h_2k}$ holds.
  \par Via that example, we also see that there do exist functions $\varphi_{\pi,e,f}$ that vanish identically on a whole left coset of $G=D_n$ with respect to $G^g=C_n$. Indeed, we have, for $0\leqslant k\leqslant n-1$,
  $$
  \varphi_{\rho^{h_1},e,f}(r^k)=0\,,
  $$
 i.e. $\varphi_{\rho^{h_1},e,f}(g)=0$ for all $g\in C_n$.
\end{remarks}

 A common feature in the two lemmas above is that each individual sum $W(\varphi,\varphi')$ involves two (a priori) distinct representations (of two a priori distinct groups). In order to estimate those sums, it will be convenient to rewrite them in such a way that a single group representation appears for each of the $W(\varphi,\varphi')$. For that purpose, we need to introduce additional notation: for $\ell$ and $\ell'$ two elements of $\Lambda$, let $G_{\ell,\ell'}$ be the group:
 $$
 G_{\ell,\ell'}=\left\{\begin{array}{l} G_\ell\times G_{\ell'} \,\,\text{if}\,\, \ell\not=\ell'\,,\\
                                        G_\ell\,,\, \text{otherwise}\,. \end{array}\right.
 $$

 Now, for $\pi$ (resp. $\tau$) an irreducible representation of $G_\ell$ (resp. $G_{\ell'}$), we define the representation of $G_{\ell,\ell'}$:
 $$
 [\pi,\tau]=\left\{\begin{array}{l} \pi\boxtimes\tau \,\,\text{if}\,\, \ell\not=\ell'\\
                                       \pi\otimes\tau\,,\, \text{otherwise}\,, \end{array}\right.
 $$
 where ``$\boxtimes$'' (resp. ``$\otimes$'') denotes the external (resp. inner) tensor product of representations. With such notation, we can give the statement of~\cite[Lemma 3.4]{KoLS} which is useful in the proofs of Propositions~\ref{borneconstconj} and~\ref{exp_sum_maj}:

 \begin{lemme}\label{lm-ortho}
  Let $\ell$, $\ell'$ in $\Lambda$, $\pi$ (resp. $\tau$) a non trivial irreducible representation of $G_\ell$ (resp. of $G_{\ell'}$).
  The multiplicity\index{multiplicity} of the trivial representation
  in the restriction of $[\pi,\bar{\tau}]$ to $G_{\ell,\ell'}^g$ is
  equal to zero if $(\ell,\pi)\not=(\ell',\tau)$, and is equal to
  $|\hat{\Gamma}_\ell^{\pi}|$ if $(\ell,\pi)=(\ell',\tau)$.
\end{lemme}

 If $\ell\not=\ell'$, it is well known that the family $([\pi,\tau])$ of representations of $G_{\ell,\ell'}$, with $\pi$ (resp. $\tau$) running over a system of representatives of irreducible representations of $G_\ell$ (resp. $G_{\ell'}$) forms itself a system of representatives for the irreducible representations of $G_{\ell,\ell'}$.
 \par
 The above notation and the sieving context we would like to work with, suggest us to combine the maps $\rho_\ell$ and $\rho_{\ell'}$ (assuming $\ell$ and $\ell'$ are distinct elements of $\Lambda$) in a single map
 \begin{align*}
 \rho_{\ell,\ell'}: G&\ra G_{\ell,\ell'}\\
                    g&\mapsto (\rho_\ell(g),\rho_{\ell'}(g))\,,
 \end{align*}
 which is nothing but the product map from $G$ to $G_\ell\times G_{\ell'}$.
 \par
  Now we claim that the exponential sums~(\ref{expW}) can be rewritten, according to the sieve setting considered, in one of the following forms:
  \begin{itemize}
  \item in the case of the conjugacy coset sieve, we have
 \begin{equation}\label{Wconjsieve}
  W(\varphi_\pi,\varphi_\tau)=\frac{1}{\sqrt{|\hat{\Gamma}_m^\pi||\hat{\Gamma}_n^\tau|}}{\bf E}(\Tr([\pi,\bar{\tau}]\rho_{\ell,\ell'}(X_k))\,,
  \end{equation}
  with notation as in Lemma~\ref{baseconjsieve},
  \item in the case of the non-conjugacy coset sieve, we have
  \begin{equation}\label{Wnonconjsieve} W(\varphi_{\pi,e,f},\varphi_{\tau,\eps,\phi})=\sqrt{\frac{{(\dim\pi)(\dim\tau)}}{|\hat{\Gamma}_m^{\varphi_{\pi,e,f}}||\hat{\Gamma}_n^{\varphi_{\tau,\eps,\phi}}|}}{\bf E}(<[\pi,\bar{\tau}](\rho_{\ell,\ell'}(X_k))\tilde{e};\tilde{f}>_{[\pi,\bar{\tau}]})\,,
 \end{equation}
 with notation as in Lemma~\ref{basenonconjsieve} and where $\tilde{e}=e\otimes\eps$, $\tilde{f}=f\otimes\phi$.
 \end{itemize}
  
  \medskip
  Both facts are a direct application of~\cite[Lemma 2.11]{KoLS}.

  \subsection{Self-contained statements}
  We finish this appendix by giving self-contained statements (i.e. using no new terminology) for Lemmas~\ref{baseconjsieve} and~\ref{basenonconjsieve} in order to make it possible for the reader to follow the whole proof of Theorem~\ref{exp_decrease} without having to get too much involved (at least for a first reading of the paper) in the details of the sieve. To begin with, we give the following self-contained version of Lemma~\ref{baseconjsieve} (this is~\cite[Prop. 3.7]{KoLS}):

  \begin{proposition}\label{pr-coset-sieve}
  Let $G$ be a group, $G^g$ a normal subgroup of $G$ with abelian
  quotient $\Gamma$; denote $d\,:\, G\ra \Gamma$ the quotient
  map. Let $\Lambda$ be a subset of the rational primes and let $\rho_{\ell}\,:\, G\ra
  G_{\ell}$, for $\ell\in\Lambda$, be a family of surjective
  homomorphisms onto finite groups. Denote
  $G^g_{\ell}=\rho_{\ell}(G^g)$. Let $\alpha\in \Gamma$ be
  fixed, $Y=d^{-1}(\alpha)\subset G$ and $Y_\ell=\rho_\ell(Y)$. Let $(\Psi,\Sigma,{\bf P})$ be a
  probability space and $X$ a random variable with values in $Y$. For any 
  
$\ell\in \Lambda$ let $\Pi_\ell$ be a set of representatives of the set of irreducible
representations of $G_\ell$ modulo equality restricted to $G^g_\ell$,
containing the constant function $1$. Moreover, let $\Pi_\ell^*=\Pi_\ell\setminus\{1\}$ and $\hat{\Gamma}_\ell^{\pi}$ be the set of characters $\psi$ of
$\Gamma_\ell=G_\ell/G^g_\ell$ such that $\pi\otimes\psi\simeq \pi$ for a
representation $\pi$ of $G_\ell$.
\par
Let $\Lstar$ be a finite subset of $\Lambda$. Then, for any conjugacy
invariant subsets $\Theta_{\ell}\subset Y_{\ell}$ for $\ell\in \Lstar$, 
we have
$$
{\bf P}(\rho_{\ell}(X)\notin \Theta_{\ell},
\text{ \emph{for all} } \ell\in\Lstar)\leqslant  \Delta H^{-1}
$$
where $\Delta$ is the smallest non-negative real number such that
$$
\sum_{\ell\in\Lstar}{\sum_{\pi\in\Pi_\ell^*}{\Bigl|
{\bf E}(\beta\cdot\Tr\pi(\rho_\ell(X)))
\Bigr|^2}}\leqslant \Delta{\bf E}(|\beta|^2)
$$
for all square-integrable functions $\beta\in L^2(\Psi,{\bf P})$, 
and 
$$
H=\sum_{\ell\in\Lstar}{\frac{|\Theta_{\ell}|}
{|G^g_\ell|-|\Theta_{\ell}|}}.
$$
\par
In addition, we have
\begin{equation}\label{eq-max-cosets}
\Delta\leqslant
\max_{\ell\in\Lstar}\max_{\pi\in\Pi_\ell^*}\sum_{\ell'\in\Lstar}\sum_{\tau\in\Pi_{\ell'}^*}
{|W(\pi,\tau)|},
\end{equation}
with
\begin{equation}\label{eq-wpitau1}
W(\pi,\tau)=\frac{1}{\sqrt{|\hat{\Gamma}_\ell^{\pi}|
|\hat{\Gamma}_{\ell'}^{\tau}|}}{\bf E}(\Tr\pi(\rho_\ell(X))
\overline{\Tr\tau(\rho_{\ell'}(X))})
=\frac{1}{\sqrt{|\hat{\Gamma}_\ell^{\pi}|
|\hat{\Gamma}_{\ell'}^{\tau}|}}{\bf E}(\Tr [\pi,\bar{\tau}](\rho_{\ell,\ell'}(X)))\,,
\end{equation}
using the notation $\rho_{\ell,\ell'}$ for the product map $\rho_\ell\times \rho_{\ell'}: G\ra G_\ell\times G_{\ell'}$ if $\ell\not=\ell'$ and $\rho_{\ell,\ell'}=\rho_\ell$ otherwise, and $[\pi,\bar{\tau}]=\pi\otimes\bar{\tau}$ for the (internal or external, depending on whether $\ell=\ell'$ or not) tensor product of the representations $\pi$ and $\bar{\tau}$.
\end{proposition}

 The analogue self-contained statement in the non-conjugacy coset sieve setting is the following reformulation of Lemma~\ref{basenonconjsieve}:

 \begin{proposition}\label{pr-group-sieve}
  Let $(G, G^g, \Lambda, (\rho_{\ell}),(G_\ell),(G_\ell^g))$, $(\Psi,\Sigma,{\bf P})$ and $\alpha, Y, (Y_\ell), X$ be as in
  Proposition~\ref{pr-coset-sieve}
  . Moreover, for each $\ell\in\Lambda$ and each
  finite dimensional irreducible representation $\pi\in {\rm Irr}(G_\ell)$ (the set of isomorphism classes of such irreducible representations), let
$$
B_{\pi}=(e_\pi^1,\ldots, e_\pi^{\dim\pi})
$$
be an orthonormal basis of the space of $\pi$ with respect to a
$G_{\ell}$-invariant inner product $\langle\,;\,\rangle_\pi$. For the set of triples $\{(\pi,e,f)\mid \pi\in{\rm Irr}(G_\ell), e,f\in B_\pi\}$, we denote by $\Pi_\ell$ a set of representatives for the equivalence relation:
$$
(\pi,e,f)\sim (\tau,\eps,\phi)\,\,\text{if}\,\, \langle\pi(g)e;f\rangle_\pi=\langle\tau(g)\eps; \phi\rangle_\tau\,,\,\text{for all}\,\, g\in G_\ell^g
$$  
such that $(1,e,e)\in \Pi_\ell$ (where $1$ denotes the trivial representation and $e$ is a basis for the $1$-dimensional space attached to it) and such that there is no $(\pi,e,f)\in \Pi_\ell$ satisfying $\langle\pi(g)e;f\rangle_\pi=0$ for all $g\in G^g$. Let $\Pi_\ell^*=\Pi_\ell\setminus\{(1,e,e)\}$ and let $\Lstar$ be a finite subset of $\Lambda$. Then, for any subsets $\Theta_{\ell}\subset Y_\ell$ for $\ell\in \Lstar$, we have
$$
{\bf P}( \rho_{\ell}(X)\notin \Theta_{\ell},
\text{ \emph{for} } \ell\in\Lstar)\leqslant  \Delta H^{-1}
$$
where $\Delta$ is the smallest non-negative real number such that
$$
\sum_{\ell\in\Lstar}{\sum_{(\pi,e,f)\in\Pi_\ell^*}{
\sqrt{\dim(\pi)}
\Bigl|
{\bf E}(\beta\cdot
\langle \pi(\rho_\ell(X))e;f\rangle)
\Bigr|^2}}\leqslant \Delta{\bf E}(|\beta|^2)
$$
for all square-integrable functions $\beta\in L^2(\Psi,{\bf P})$, where
$$
H=\sum_{\ell\in\Lstar}{\frac{|\Theta_{\ell}|}
{|G_{\ell}|-|\Theta_{\ell}|}}.
$$

 Moreover we have
 \begin{equation}\label{eq-max-coset-nc}
 \Delta\leqslant \max_{\ell\in\Lstar}\max_{(\pi,e,f)\in\Pi_\ell^*}\sum_{\ell'\in\Lstar}{\sum_{(\tau,\eps,\phi)\in\Pi_{\ell'}^*}{|W((\pi,e,f),(\tau,\eps,\phi))|}}\,,
 \end{equation}
 where, with the same notations as in~(\ref{eq-wpitau1}),
 \begin{align}\label{eq-wpitau2} W((\pi,e,f),(\tau,\eps,\phi))&=\sqrt{\frac{{(\dim\pi)(\dim\tau)}}{|\hat{\Gamma}_\ell^{(\pi,e,f)}||\hat{\Gamma}_{\ell'}^{(\tau,\eps,\phi)}|}}{\bf E}(\langle\pi(\rho_\ell(X))e;f\rangle_\pi\overline{\langle\tau(\rho_\ell'(X))\eps;\phi\rangle_\tau})\\
 &=\sqrt{\frac{{(\dim\pi)(\dim\tau)}}{|\hat{\Gamma}_m^{\varphi_{\pi,e,f}}||\hat{\Gamma}_n^{\varphi_{\tau,\eps,\phi}}|}}{\bf E}(\langle[\pi,\bar{\tau}](\rho_{\ell,\ell'}(X))(e\otimes\eps);(f\otimes\phi)\rangle_{[\pi,\bar{\tau}]})\nonumber\,.
 \end{align}
with $\hat{\Gamma}_\ell^{(\pi,e,f)}$ denoting the set of characters $\chi$ of $\Gamma_\ell$ such that
$$
\langle\pi(g)e;f\rangle_\pi\cdot \chi(g)=\langle\pi(g)e;f\rangle_\pi\,.
$$

\end{proposition}

 \begin{remark}
    In our description of coset sieves, we have restricted ourselves to sieve supports containing only \emph{prime} numbers. Nevertheless as suggested by the discussions preceding and following Lemma~\ref{lm-ortho}, we could quite easily extend our sieve method to a framework in which we would use squarefree integers (and not only primes) as a sieve support. As described in~\cite[Chap. 3]{KoLS}, going from a prime sieve support to a ``squarefree'' sieve support can be done naturally by extending a few of the definitions we have given in this appendix by multiplicativity. Although using that extended sieve support would surely yield better estimates in Theorem~\ref{exp_decrease}, we prefer working only with a prime sieve support, so that we avoid the use of additional notation. However, for the proof of Proposition~\ref{exp_sum_maj}, it is convenient to use objects defined by multiplicativity from two (not more) primes in $\Lambda$. So, for $\ell\not=\ell'$ two such primes, let
 $$
 Y_{\ell,\ell'}=Y_\ell\times Y_{\ell'}\,,
 $$
 on which we have the product density $\nu_{\ell,\ell'}(y,y')=\nu_\ell(y)\nu_{\ell'}(y')$, so that it makes sense to speak about the space $L^2(Y_{\ell,\ell'},\nu_{\ell,\ell'})$. It is straightforward to check that if $\mt{B}_\ell$ (resp. $\mt{B}_{\ell'}$) is an orthonormal basis of $L^2(Y_\ell,\nu_\ell)$ (resp. of $L^2(Y_{\ell'},\nu_{\ell'})$), the family of functions defined by $(y,y')\in Y_{\ell,\ell'}\mapsto \varphi(y)\varphi'(y')$, where $\varphi\in\mt{B}_\ell$ and $\varphi'\in\mt{B}_{\ell'}$, forms an orthonormal basis of $L^2(Y_{\ell,\ell'},\nu_{\ell,\ell'})$. Note finally, that, to unify all the possible cases, we can extend the above definitions to the case $\ell=\ell'$ by defining $Y_{\ell,\ell'}=Y_\ell$, $\nu_{\ell,\ell'}=\nu_\ell$ and $\mt{B}_{\ell,\ell'}=\mt{B}_\ell$.
  \end{remark}


\begin{thebibliography}{CCC}

\bibitem[Ar.E]{Ar.E}
E. Artin : \textit{Geometric algebra}, Interscience Tracts in Pure and Appl. Math. 3, Interscience publishers, New York, (1957).






\bibitem[ABS]{ABS}
M. F. Atiyah, R. Bott and A. Shapiro : \textit{Clifford modules}, Topology $3$ (1964) suppl. $1$, 3-38.




\bibitem[Ba]{Ba}
R. Baeza : \textit{Discriminants of polynomials and of quadratic forms}, J. Algebra 72, (1981), 17--28.






\bibitem[Bo]{Bo}
A. Borel : \textit{Linear algebraic groups}, Second edition. Graduate Texts in Mathematics, 126. Springer-Verlag, New York, (1991).



\bibitem[Ca]{Ca}
L. Carlitz: \textit{Some theorems on irreducible reciprocal polynomials over a finite field}, J. Reine Angew. Math. 227 (1967) 212--220.



\bibitem[Chav]{Ch}
N. Chavdarov: \textit{The generic irreducibility of the numerator of
  the zeta function in a family of curves with large monodromy}, Duke
Math. J. 87 (1997), 151--180.


\bibitem[Chung]{C}
F. K. R. Chung : \textit{Diameters and eigenvalues}, J. Amer. Math. Soc. 2, 187-196, (1989).


\bibitem[DeW]{deligne_weilII}
P. Deligne: \textit{La conjecture de Weil II}, Inst. Hautes �tudes Sci. Publ. Math. No. 52 (1980), 137--252.

\bibitem[De]{De}
P. Deligne: \textit{Cohomologie \'etale}, S.G.A 4$\frac{1}{2}$, L.N.M. 569,
Springer Verlag 1977.


\bibitem[D]{D}
J. Dieudonn\'e : \textit{Sur les groupes classiques}, Hermann, (1958).


\bibitem[E]{E}
B. H. Edwards : \textit{Rotations and discriminants of quadratic spaces}, Lin. and Multilin. Algebra 8, (1980), 241--246.








\bibitem[HM]{HOM}
A. J. Hahn and O. T. O'Meara : \textit{The classical groups and
  $K$-theory}, Grundlehren der math. Wiss. 291, Springer-Verlag
1989. 




\bibitem[HV]{HV}
P. de la Harpe and A. Valette : \textit{La propri\'et\'e (T) de
  Kazhdan pour les groupes localement compacts}, Ast\'erisque 175,
S.M.F (1989). 


\bibitem[Hu]{Hu}
J. E. Humphreys : \textit{Linear algebraic groups}, Springer-Verlag, New York, 1975, GTM 21.










\bibitem[IK]{IK}
H. Iwaniec and E. Kowalski : \textit{Analytic Number Theory},
Colloquium Publ. 53, A.M.S 2004. 





\bibitem[JKZ]{JKZ}
F. Jouve, E. Kowalski and D. Zywina: \textit{An explicit integral polynomial whose splitting field has Galois group $W(E8)$}, submitted.






\bibitem[KaMul]{KaM}
N. M. Katz : \textit{Estimates for nonsingular multiplicative character sums}, Int. Math. Res. Not. (2002), no. 7, 333--349.

\bibitem[KaL]{Ka}
N. M. Katz : \textit{Report on the irreducibility of $L$-functions}, to appear, (volume in honour of S. Lang).



\bibitem[Kn]{Kn}
A. Knapp : \textit{Representation theory of semisimple groups},
Princeton Math. Series 36, Princeton Univ. Press, 1986.




\bibitem[KoZeta]{KoCrelle}
E. Kowalski : \textit{The large sieve, monodromy and zeta functions of
  curves}, J. reine angew. Math 601 (2006), 29--69.


\bibitem[KoDef]{KoIs}
E. Kowalski : \textit{Exponential sums over definable subsets of finite fields}, Israel J. Math. 160 (2007), 219--251.

  

\bibitem[KoSieve]{KoLS}
E. Kowalski : \textit{The large sieve and its applications}, Cambridge Tracts in Math. 175, Cambridge Univ. Press, and \textit{The principle of the large sieve}, \url{arXiv:math.NT/0610021}






\bibitem[La]{La}
S. Lang : \textit{Algebraic groups over finite fields}, Amer. J. Math. 78 (1956), 555--563. 





\bibitem[LPS]{LPS}
A. Lubotzky, R. Philipps and P. Sarnak : \textit{Ramanujan graphs}, Combinatorica 8, no 3, 261-277, (1988).


\bibitem[Lu]{Lu}
A. Lubotzky : \textit{Discrete groups, expanding graphs and invariant
  measures}, Progr. Math. 125, Birkh\"auser 1994.


\bibitem[LZ]{LZ}
A. Lubotzky and A. \.Zuk : \textit{On Property $(\tau)$}, draft, online.







\bibitem[Me]{meyn}
H. Meyn : \textit{On the construction of irreducible self-reciprocal polynomials over finite fields},  Appl. Algebra Engrg. Comm. Comput.  1  (1990),  no. 1, 43--53. 


\bibitem[Mo]{M}
M. Morgenstern : \textit{Existence and explicit constructions of $q+1$ regular Ramanujan graphs for every prime power $q$}, J. Combin. Theory Ser. B 62, no. 1, 44-62  (1994).


\bibitem[No]{No}
M. Nori : \textit{On subgroups of ${\rm GL}_n(\F_p)$},  Invent. Math.  88  (1987),  no. 2, 257--275.




\bibitem[OM]{OM}
O. T. O'Meara : \textit{Introduction to Quadratic Forms}, Grundlehren der math. Wiss. 117, Springer-Verlag, (1963).


\bibitem[PR]{PR}
V. Platonov and A. Rapinchuk : \textit{Algebraic groups and Number Theory}, Pure and Appl. Math., Acad. Press 139, (1994).









\bibitem[Ser]{Se}
J. P. Serre : \textit{Repr\'esentations lin\'eaires des groupes finis}, Coll. M\'ethodes, Hermann, (1967). 


\bibitem[Sp]{Sp}
T. A. Springer : \textit{Conjugacy classes in algebraic groups}, Group theory, Beijing (1984), Lecture Notes in Math., 1185, 175--209.



\bibitem[Sh]{shimura}
G. Shimura : \textit{Introduction to the arithmetic theory of automorphic functions}, Princeton Univ. Press (1971).











\bibitem[SpSt]{SpSt}
T. A. Springer and R. Steinberg : \textit{Conjugacy classes}, Seminar on Algebraic Groups and Related Finite
Groups, Lecture Notes in Mathematics, 131, (1968-69), 167-266, Springer-Verlag, Berlin-New York,
1970. 



\bibitem[Ste1]{St1}
R. Steinberg :
\textit{Endomorphisms of linear algebraic groups},
Memoirs of the AMS, No. 80
AMS, Providence, R.I. (1968).


\bibitem[Ste2]{St2}
R. Steinberg : \textit{Conjugacy classes in algebraic groups}, LNM, Vol. 366. Springer-Verlag, Berlin-New York, (1974).








\bibitem[Za]{Za}
H. Zassenhaus : \textit{On the spinor norm}, Arch. Math. 13 (1962), 434--451. 


\end{thebibliography}
\end{document}